\documentclass[reqno,oneside]{amsart}
\usepackage{amsthm,amssymb, amsmath, amscd}

\usepackage[usenames, dvipsnames]{color}
\usepackage[pdfborder={0 0 0}, colorlinks = true, linkcolor=blue, citecolor=OliveGreen]{hyperref}

\usepackage{stmaryrd}
\usepackage{tikz-cd}
\usepackage{cleveref}
\usepackage{enumerate}
\usepackage{mathtools}
\usepackage{todonotes}

\usepackage{geometry}

\theoremstyle{plain}
\newtheorem{theorem}{Theorem}
\newtheorem{lemma}[theorem]{Lemma}
\newtheorem{corollary}[theorem]{Corollary}
\newtheorem{definition}[theorem]{Definition}

\newtheorem{proposition}[theorem]{Proposition}

\theoremstyle{remark}
\newtheorem{remark}[theorem]{Remark}

\numberwithin{theorem}{section}

\newcommand{\bbA}{\mathbb A}
\newcommand{\bbB}{\mathbb B}
\newcommand{\bbC}{\mathbb C}
\newcommand{\bbD}{\mathbb D}
\newcommand{\bbF}{\mathbb F}
\newcommand{\bbH}{\mathbb H}

\newcommand{\bbQ}{\mathbb Q}
\newcommand{\bbR}{\mathbb R}

\newcommand{\bbX}{\mathbb X}

\newcommand{\bbZ}{\mathbb Z}

\newcommand{\calB}{\mathcal B}
\newcommand{\calC}{\mathcal C}
\newcommand{\calD}{\mathcal D}
\newcommand{\calE}{\mathcal E}
\newcommand{\calF}{\mathcal F}

\newcommand{\calL}{\mathcal L}
\newcommand{\calM}{\mathcal M}

\newcommand{\calO}{\mathcal O}
\newcommand{\calP}{\mathcal P}
\newcommand{\calS}{\mathcal S}
\newcommand{\calV}{\mathcal V}

\newcommand{\calX}{\mathcal X}
\newcommand{\calY}{\mathcal Y}
\newcommand{\calZ}{\mathcal Z}

\newcommand{\psm}[1]{\left( \begin{smallmatrix}#1 \end{smallmatrix} \right)}

\newcommand{\nm}{\mathrm{nm}}

\newcommand{\Nilp}{\mathrm{Nilp}}
\newcommand{\Spf}{\mathrm{Spf}}
\newcommand{\Spec}{\mathrm{Spec}}
\newcommand{\Lie}{\mathrm{Lie}}

\newcommand{\Hom}{\mathrm{Hom}}
\newcommand{\ChowHat}[1]{\widehat{\mathrm{CH}}{}^{#1}}

\newcommand{\ddc} {\mathrm{dd^c}}

\newcommand{\Dpre}{\calD_{\mathrm{pre}}}

\DeclareMathOperator{\End}{End}
\DeclareMathOperator{\Aut}{Aut}
\DeclareMathOperator{\Sym}{Sym}
\DeclareMathOperator{\Sp}{Sp}
\newcommand{\GL}{\mathrm{GL}}
\newcommand{\ord}{\mathrm{ord}}
\DeclareMathOperator{\Mp}{Mp}

\newcommand{\A}{\mathbb{A}}

\newcommand{\Z}{\mathbb{Z}}

\newcommand{\C}{\mathbb{C}}

\title{A genus two arithmetic Siegel-Weil formula on $X_0(N)$}
\author{Siddarth Sankaran, Yousheng Shi, Tonghai Yang}
\date{}

	
	\thanks{SS was partially supported by an NSERC Discovery grant. TY was partially supported by Van Vleck Research grant and Dorothy Gollmar chair fund}
	
\begin{document}
	\maketitle
	
	\begin{abstract}
		We define a family of arithmetic zero cycles in the arithmetic Chow group of a modular curve $X_0(N)$, for $N>3$ odd and squarefree, and identify the arithmetic degrees of these cycles as $q$-coefficients of the central derivative of a Siegel Eisenstein series of genus two. This parallels work of Kudla-Rapoport-Yang for Shimura curves.
	\end{abstract}

	\tableofcontents
	%############################################################################################

	\section{Introduction}
	
	In a series of work culminating in the book \cite{KRYbook}, Kudla, Rapoport and the third named author studied certain families of arithmetic ``special" cycles that live on arithmetic models of Shimura curves. Among their results are  \emph{arithmetic Siegel-Weil formulas} in genus one and two, which identify generating series of heights of arithmetic  cycles with derivatives of Eisenstein series. The results of \cite{KRYbook} comprise the most fully-developed example in Kudla's program, which seeks to establish systematic relations between arithmetic cycles on Shimura varieties and automorphic forms; while a substantial body of work has arisen in support of these conjectures, the case of modular curves has been largely overlooked.
	
	In this note, we fill this gap in the literature and prove the arithmetic Siegel-Weil formula for arithmetic zero cycles on $\calX_0(N)$ for odd, squarefree $N$ with $N>3$. More precisely, we construct a family of arithmetic zero cycles $\widehat \calZ(T,v)$, viewed in the arithmetic Chow group $\ChowHat{2}(\calX_0(N))$. Our main result, explicated in more detail in \Cref{sect2.5} and \Cref{sec:main theorem setup} below, is the identity
	\begin{equation}
	\sum_{T \in \Sym_2(\mathbb Z)} \, \widehat{\mathrm{deg}} \,  \widehat\calZ(T,v) \, q^T  \  = \ C \cdot E'(\tau, 0, \Phi^{\calL});
	\end{equation}
	here
	\begin{itemize}
		\item $\tau \in \mathbb H_2$, the Siegel upper half space of genus 2 and $v = \mathrm{Im}(\tau)$;
		\item $q^T = e^{2 \pi i \mathrm{tr}(\tau T)}$;
		\item  $C = \frac1{24}\prod_{p|N}(p+1)$;
		\item $\widehat{ \mathrm{deg}}$ is the arithmetic degree; and
		\item  $E'(\tau, 0, \Phi^{\calL})$ is the derivative of the Siegel Eisenstein series, evaluated at the centre of symmetry $s=0$, associated to the quadratic lattice $\calL$  given by
			\begin{equation}
		\calL= \left\{x=\begin{pmatrix} a &b/N \\ c & -a \end{pmatrix}:\,  a, b, c  \in \Z\right\},  \quad  Q_\calL(x)=N det (x) = -Na^2 -bc.
		\end{equation}
	\end{itemize}

	Our proof of the main result, which proceeds by matching the $q$-coefficients term-by-term, is closely modelled on the arguments in  \cite{KRYbook}. For positive definite $T$, it turns out that the cycle $\widehat{\calZ}(T,v)$ is concentrated in a special fibre at a prime $p$, and the geometric side of the theorem factors as the product of a local intersection number, determined by Gross and Keating \cite{GK}, and a point count; we then relate this computation to the corresponding coefficient of the Eisenstein series using explicit formulas due to the third named author.   An important step is to prove a local arithmetic Siegel-Weil formula ( \Cref{prop:local intersection number}), which  identifies the local intersection number with the central derivative of a local  Whittaker function.  For non-degenerate $T$ of signature $(1,1)$ or $(0,2)$, the special cycles are purely archimedean, and the result is essentially a special case of \cite[Theorem 5.1]{GarciaSankaran}. Finally, for degenerate $T$'s,   the identity can be reduced to the genus one case of the Siegel-Weil formula, as considered in \cite{DuYang}.
	
	Finally, we note that there is another quadratic lattice $L$ that is naturally associated to $X_0(N)$, which  is given  by
	\begin{equation}
	L=\left\{x=\begin{pmatrix} a &b \\ Nc & -a \end{pmatrix}:\,  a, b, c  \in \Z  \right\},  \quad  Q_L(x)= det(x) = -a^2 - N bc.
	\end{equation}
	The precise relation between $E'(\tau, 0, \Phi^{\calL})$ and $E'(\tau, 0, \Phi^{L})$  is given by  \Cref{prop5.3} and \eqref{eqUp1}. For the purposes of the computations in the present paper, it seems that the lattice $\mathcal L$ is more convenient.

	%############################################################################################
	\section{Arithmetic special cycles on modular curves}
	\subsection{Modular curves}
	Throughout, we fix $N>3$ odd and squarefree. Let
	\begin{equation}
	\calX = \calX_0(N)
	\end{equation}
	denote the integral model, over $\Spec(\bbZ)$ of the modular curve of level $\Gamma_0(N)$, as in \cite{KatzMazur}. More precisely, $\calX$ denotes the Deligne-Mumford stack over $\mathrm{Spec}(\bbZ)$ whose $S$ points, for a scheme $S$ over $\Spec(\bbZ)$, parametrize diagrams
	\begin{equation}
	\varphi \colon \calE \to \calE'
	\end{equation}
	where $\calE$ and $\calE'$ are generalized elliptic curves over $S$, and $\varphi$ is a cyclic isogeny of degree $N$. As usual, we let $\calY = \calY_0(N)$ denote the open modular curve $ \calY = \calX \setminus \{ \text{cusps} \}$.
	
	For later use, we recall the following description of the complex points $Y_0(N) = \calY(\bbC)$ as an $O(1,2)$  Shimura variety.	Consider the quadratic space
	\begin{equation} \label{eqn:calV defn}
	\calV := M_2(\bbQ)^{tr = 0}, \qquad Q(x) = N \det(x)
	\end{equation} 	
	of signature $(1,2)$, and  let
	\begin{equation}
	\bbD = \bbD(\calV) \ = \ \{ z \in \calV\otimes_{\bbQ}\bbC \ | \ \langle z, z \rangle = 0, \langle z, \overline z \rangle < 0  \} \big/ \bbC^{\times}
	\end{equation}
	denote the symmetric space attached to $\calV \otimes \bbC = M_2(\bbC)^{tr=0}$; here $\langle x, y \rangle = -N \mathrm{tr}(x y)$ is the complex bilinear form with $\langle x, x \rangle = 2 Q(x) = 2N\det x$. The space $\mathbb D$ decomposes into two connected components
	\begin{equation}
	\bbD = \bbD^+ \coprod \bbD^-,
	\end{equation}
	where
	\begin{equation}
	\bbD^+  := \left\{  z = \psm{ a& b \\ c & -a } \in \calV \otimes \bbC \ | \ \langle z, z \rangle = 0, \, \langle z, \overline z \rangle < 0 , \, \mathrm{Im}(a/c) > 0  \right\} \big/ \bbC^{\times}.
	\end{equation}
	The group $\GL_2(\bbR)$ acts on $\calV \otimes \bbC$ by $\gamma \cdot v = \gamma v \gamma^{-1}$, which descends to an action on $\bbD$. A straightforward computation shows that $\bbD^+$ is invariant under the action of $\mathrm{SL}_2(\bbR)$, and the matrix $\mathrm{diag}(1,-1)$ interchanges the components.
	
	Moreover, we have a $\mathrm{SL}_2(\bbR)$-equivariant identification
	\begin{equation} \label{eqn:upper half plane sym space isom}
	\bbH  \stackrel{\sim}{\longrightarrow} \bbD^+, \qquad \tau \mapsto \mathrm{span}_{\bbC} \, z(\tau)
	\end{equation}
	where
	\begin{equation}
	z(\tau) = \begin{pmatrix}\tau  & -\tau^2 \\ 1 & - \tau \end{pmatrix} .
	\end{equation}
	This construction gives rise to identifications
	\begin{equation}
	\calY(\bbC) \simeq [\Gamma_0(N) \big\backslash \bbH] \simeq [\Gamma_0(N) \big \backslash \bbD^+].
	\end{equation}
	Note here we are viewing the quotients as orbifolds.
	
	\subsection{Special cycles}
	We begin by recalling the construction of special cycles, following  \cite{DuYang} or \cite{BruinierYang}.
	For a point $(\varphi \colon E \to E') \in \calX_0(N)(S)$, for some base scheme $S$, let
	\begin{equation}
	\End(\varphi) = \{  x \in \End(E) \ | \ \varphi \circ x \circ \varphi^{-1} \in \End(E')  \}
	\end{equation}	
	and define
	\begin{equation} \label{eqn:moduli lattice def}
	L_{\varphi} \ := \ \left\{ \alpha \in \End(\varphi) \ | \ \alpha + \alpha^{\dagger} = 0 \right\},
	\end{equation}
	where $\alpha^{\dagger}$ is the Rosati dual. We may equip $L_{\varphi}$ with the positive definite quadratic form
	\begin{equation}
	Q_{\varphi}(\alpha) = \deg (\alpha) = -  \alpha^2.
	\end{equation}

	\begin{definition} \label{def:special divisors moduli} For $m\in \bbZ$, let $\calZ(m)$ denote the moduli stack whose $S$ points, for a base scheme $S$, are given by	
		\begin{equation}
		\calZ(m)(S) := \left\{	(\varphi \colon E \to E', \alpha) \right\}
		\end{equation}
		where $\alpha \in L_{\varphi}$ satisfies  the following conditions:
		\begin{enumerate}[(a)]
			\item $Q_{\varphi}(\alpha) = mN$;
			and
			\item  the quasi-isogeny $\alpha \circ \varphi^{-1}$ is in fact an isogeny from $E'$ to $E$.
		\end{enumerate}
	\end{definition}

	\begin{remark}
		\begin{enumerate}[(i)]
			\item
		More generally, in  \cite{DuYang,BruinierYang} one finds a definition for cycles $\calZ(m, \mu_r)$ parametrized by $m \in \frac{1}{4N}\bbZ$ and $r \in \bbZ/2N \bbZ$ satisfying $r^2 \equiv -4 Nm \pmod{4N}$; for $m \in \bbZ$, the moduli problem for $\calZ(m)$ described above coincides with $\calZ(m, \mu_0)$ in the notation of loc.\ cit.\
			\item If $\alpha \in L_{\varphi}$ satisfies condition (b) above, then  $Q_{\varphi}(\alpha)$ is divisible by $N$. Moreover, if one omits this condition, the resulting moduli problem does not define a divisor on $\mathcal X_0(N)$.
		\end{enumerate}
	\end{remark}

	There is a natural forgetful map $\calZ(m) \to \calX_0(N)$, which allows us to view $\calZ(m)$ as a cycle on $\calX_0(N)$ (which, abusing notation, we denote by the same symbol).
	
	\begin{lemma}
		The cycle $\calZ(m)$ is a divisor on $\calX_0(N)$. Moreover, it is equal to the flat closure of its generic fibre $Z(m) = \calZ(m)_{/\bbQ}$. 			\begin{proof}
			We begin by showing that $\calZ(m)$ is a divisor. This is clear on the generic fibre, so it will suffice to verify this claim in an \'etale neighbourhood of a point in characteristic $p$. To this end, let $z \in \calZ(m)(\bar{\bbF}_p)$, lying above the point $x \in \calX_0(N)(\bar{\bbF}_p)$, and let $R_z$ (resp.\ $R_x$) denote the complete \'etale local rings of $z$ and $x$ respectively.
			Thus we have a surjection
			\begin{equation}
			R_x \to R_z = R_x / J
			\end{equation}			
			for some ideal $J$ that we would like to show is principal. By Nakayama's lemma, it will suffice to show that $I = J/mJ$ is principal, where $m$ is the maximal ideal of $R_x$. Note that $I^2 = 0$, so that the surjection
			\begin{equation}
			A := R_x / mJ \to R_z
			\end{equation}
			is a square-zero infinitesimal extension.
			
			Let $\widetilde \varphi \colon \calE \to \calE'$ denote the universal diagram over $R_x$, and $\widetilde \varphi_A \colon \calE_A \to \calE'_A$ (resp.\ $\widetilde \varphi_z \colon \calE_{R_z} \to \calE_{R_z}'$) its base change to $A$ (resp.\ $R_z$).
			
			By assumption, we have an action $\alpha_z \in \End(\widetilde \varphi_z)$ such that
			\begin{equation}
			\delta :=  \alpha_z \circ \widetilde \varphi_z^{-1}
			\end{equation} is an isogeny $\calE'_{R_z} \to \calE_{R_z}$. Moreover, the ideal $I \subset A$ is, by definition, the largest ideal such that $\delta$ deforms to a homomorphism  in $\Hom(\calE'_{A/I}, \calE _{A/I})$.
			
			Now we apply Grothendieck-Messing theory. Consider the Hodge exact sequences
			\begin{equation}
			0 \to \calF(\calE'_A) \to H^1_{dR}(\calE'_A) \to \Lie(\calE'_A) \to 0 \qquad \text{and} \qquad 					0 \to \calF(\calE_A) \to H^1_{dR}(\calE_A) \to \Lie(\calE_A) \to 0
			\end{equation}
			for $\calE'_A$ and $\calE_A$ respectively.  Since the map $A \to R_z$ is a square-zero thickening, the homomorphism $\delta \colon \calE'_{R_z} \to \calE_{R_z}$ induces a canonical $A$-linear map
			\begin{equation}
			H^1_{dR}(\calE_A') \to H^1_{dR}(\calE_A);
			\end{equation}	
			composing this with the maps in the Hodge exact sequences above, we obtain a map
			\begin{equation}
			\widehat \delta \colon \calF(\calE'_A) \to \Lie(\calE_A).
			\end{equation}
			Then Grothendieck-Messing theory implies that for any intermediate ring $B$ with
			\begin{equation}
			A \to B \to R_z
			\end{equation}
			the homomorphism $\delta $ lifts to an element of $\Hom(\calE'_B , \calE_B)$ if and only if $\widehat \delta \otimes B = 0$. This immediately implies that the ideal $I$ is given by the vanishing locus of $\widehat \delta$, and since $\calF(\calE'_A)$ and $\Lie(\calE_A)$ are free $A$-modules of rank 1, it follows that $I$ is principal, as required.
			
			It remains to show that $\calZ(m)$ is horizontal, i.e. does not contain any irreducible components in finite characteristic. Let $p$ be a prime, and suppose $z \in \calZ(m)(\bar{\bbF}_p)$   corresponds to $(\varphi\colon E \to E', \alpha)$ as before. If $p$ splits in $\bbQ(\alpha) = \bbQ(\sqrt{-mN})$, then $E$ and $E'$ are necessarily ordinary; if $p$ is non-split, then $E$ and $E'$ are both supersingular. However, every irreducible component of $\calX_0(N)(\bar{\bbF}_p)$ contains both ordinary and supersingular points, hence such a component cannot be contained in the support of $\calZ(m)$.
		\end{proof}
	\end{lemma}
	
	We will also require a description of the complex points $\calZ(m)(\bbC)$. Consider the quadratic lattice 	
	\begin{equation} \label{eqn:lattice def}
	\calL := \left\{  \begin{pmatrix} a & b/N \\ c & -a\end{pmatrix} \Big| a,b,c\in \bbZ \right\} \ \subset \ \calV
	\end{equation}
	where $\calV = M_2(\bbQ)^{tr = 0}$, with quadratic form $Q(x) = N \det x$. 	
	
	Now given $x  = \psm {a & b/N \\ c & -a } \in \calL$ with $x \neq 0$, let
	\begin{equation}
	\bbD^+_x \ := \ \{ z \in \bbD^+ \ | \ \langle z, x \rangle = 0  \}.
	\end{equation}
	If $Q(x) \leq 0$, then $\bbD^+_x$ is  empty. Otherwise, $\bbD^+_x = \{  z(\tau_x ) \}$  consists of exactly one point, where $\tau_x \in \bbH$ is the root of the equation
	\begin{equation}
	c \tau^2 + 2 a \tau - b/N
	\end{equation}	
	with positive imaginary part.
	
	\begin{lemma} We have identifications
		\begin{equation}
		\begin{tikzcd}
		\mathcal{Z}(m)(\mathbb C) \arrow[d] \arrow[rr, "\sim"] &  & { \displaystyle \left[ \Gamma_0(N) \big\backslash \coprod_{\substack{x \in \calL \\ Q(x) = m}} \bbD^+_x \right]} \arrow[d]     \\
		\mathcal{Y}_0(N)(\mathbb C) \arrow[rr, "\sim"]         &  & {\left[ \Gamma_0(N) \big\backslash \mathbb H \right]}
		\end{tikzcd}
		\end{equation}
		\begin{proof}
			Given $x = \psm{a & b/N \\ c & -a } \in \calL$, let $\tau=\tau_x$ as above. Then the corresponding point of $\calY_0(N)(\bbC)$ is the diagram
			\begin{equation} \label{eqn:complex moduli point}
			\varphi_{\tau} \colon E_{\tau} = \bbC /  \Lambda_{\tau}   \ \stackrel{\times N}{\longrightarrow}  \ E'_{\tau} = \bbC / \Lambda_{N\tau}
			\end{equation}
			where $\varphi_{\tau}$ is multiplication by $N$, and $\Lambda_{\tau} = \bbZ + \bbZ\tau$ for $\tau \in \bbH$. Define an endomorphism $\alpha \in \End(\varphi_{\tau})$ by
			\begin{equation}
			\alpha(z + \Lambda_{\tau}) = z  (Na + Nc \tau) + \Lambda_{\tau}.
			\end{equation}
			It is straightforward to check that $\alpha$ satisfies the conditions in \Cref{def:special divisors moduli},
			and every such endomorphism arises in this way.
		\end{proof}
	\end{lemma}
	
	Next, we turn our attention to 0-cycles. Let $\Sym_2(\bbZ)^{\vee} = \{ \psm{a & b/2 \\ b/2 & c} | a,b,c \in \bbZ  \}$, suppose
	\begin{equation}
	T = \begin{pmatrix} t_{11} & * \\ * & t_{22} \end{pmatrix} \in \Sym_{2}(\bbZ)^{\vee}
	\end{equation}
	and assume for the moment that $\det T \neq 0$.
	
	\begin{definition} Let  $\calZ(T)$ denote the moduli stack whose $S$ points parametrize tuples
		\begin{equation}
		\calZ(T)(S) \ := \ \left\{  (\varphi\colon E \to E', \alpha_1, \alpha_2 )\right\}
		\end{equation}
		where
		\begin{enumerate}[(a)]
			\item $\alpha_i \in L_{\varphi}$ for $i=1,2$;
			\item the moment matrix
			$$
			T_{\varphi}(\alpha_1, \alpha_2) =( \frac{1}2 (\alpha_i, \alpha_j))= \begin{pmatrix} - \alpha_1^2 & -\frac12\left( \alpha_1 \alpha_2 + \alpha_2 \alpha_1\right) \\ -\frac12\left( \alpha_1 \alpha_2 + \alpha_2 \alpha_1\right) & - \alpha_2^2  \end{pmatrix}
			$$
			satisfies
			$$
			T_{\varphi}(\alpha_1, \alpha_2) = NT,
			$$
			\item	and finally, that $\alpha_i \circ \varphi^{-1} \in \Hom(E', E)$ for $i=1,2$.
		\end{enumerate}
	\end{definition}
	
	To describe the geometric points of this stack, we first recall the following representability criteria. Let $V / \mathbb Q_{\ell}$  be a quadratic space of dimension $m = 3$. Define the character
	\begin{equation}
	\chi_{V,\ell}(x)  = (x, (-1)^{m(m-1)/2} \det V)_{\ell} =  (x, - \det V)_{\ell}.
	\end{equation}
	
	For a given choice of determinant $\det(V) \in \bbQ_{\ell}^{\times} / \bbQ_{\ell}^{\times, 2}$, there are exactly two non-isomorphic quadratic spaces $V^{\pm}_{\ell}$, distinguished by their Hasse invariants.
	
	Now if $T \in \Sym_2(\bbQ_{\ell})$ is non-degenerate, we have that $T$ is represented by $V_{\ell}$ if and only if
	\begin{equation} \label{eqn:Hasse criterion}
	\epsilon(V) =  \epsilon(T)  \, \chi_V(\det T)
	\end{equation}
	\cite[Prop.\ 1.3]{KudlaCentralDerivs}.

	Returning to the quadratic space $\calV$ as in \eqref{eqn:calV defn}, define the \emph{difference set} $\mathrm{Diff}(T,\calV)$ to be
		\begin{align}
			\mathrm{Diff}(T,\calV) &= \{  \ell \text{ finite prime} \ | \  T \text{ is not represented by } \calV_{\ell}    \} \\
			&= \{  \ell \ | \ \epsilon_{\ell}(T) = - (-1, N)_{\ell} \cdot (\det(T), -N)_{\ell}   \}.
		\end{align}

	\begin{lemma} \label{lemma:support of Z(T)}
		Suppose $T$ is non-degenerate.
		\begin{enumerate}
			\item 	If $T$ is not positive definite, or if $T>0$ and $\# \mathrm{Diff}(T,\calV) \geq 2$, then $\calZ(T) = \emptyset$.
			\item If $T >0$ and $\mathrm{Diff}(T,\calV) = \{ p \}$, then $\calZ(T)$ is supported in the supersingular locus in the special fibre $\calX_0(N)_{/ \bbF_{p}}$. In particular, $\calZ(T) \to \calX_0(N)$ determines a zero cycle.
		\end{enumerate}
		\begin{proof}
			Suppose we have a geometric point $z \in \calZ(T)(\bbF)$ for an algebraically closed field $\bbF$, corresponding to a tuple $\varphi \colon E \to E'$. Let $V_{\varphi} = \End(\varphi)^{tr= 0} \otimes \bbQ$,  equipped with the quadratic form $x \mapsto N^{-1} \deg x$. By assumption, we have that $V_{\varphi}$ represents $T$. However, $V_{\varphi}$ is positive definite, so if $T$ is not positive definite, we obtain a contradiction; thus $\calZ(T) = \emptyset$ if $T$ is not positive definite.
			
			Assuming that $T>0$ and continuing, we note that $\dim V_{\varphi} \geq 3$, so the characteristic of $\bbF$ is non-zero and $E, E'$ are supersingular elliptic curves. Setting $p = \mathrm{char}(\mathbb F)$, we may identify $V_{\varphi} = B^{tr = 0}$, where $B$ is the rational quaternion algebra ramified at $p$ and $\infty$, equipped with the quadratic form $N^{-1} \nm$, where $\nm(x)$ is the reduced norm. It follows that for $\ell \neq p$, the space $V_{\varphi,\ell}$ is isometric to $ \calV_{\ell} $, and so $\mathrm{Diff}(T, \calV) \subset \{ p \}$.
			
			Finally, note that $\det(V_{\varphi,p})= \det(\calV_p)$ and  $\epsilon_p(V_{\varphi,p}) = - \epsilon_p(\calV_p)$. Hence, applying the representability criterion \eqref{eqn:Hasse criterion}, we conclude that $p \in \mathrm{Diff}(T,\calV)$. As we have already seen that any geometric point of $\calZ(T)$ is in the supersingular locus, we have concluded the proof of the lemma.
		\end{proof}
	\end{lemma}
	
	\subsection{Classes in arithmetic Chow groups} \label{sec:arith Ch gps}
	
	Let
	\begin{equation}
	\ChowHat{\bullet}_{\bbR}(\calX ) =  \bigoplus_{n=0}^2	\ChowHat{n}_{\bbR}(\calX)
	\end{equation}
	denote the arithmetic Chow ring, as originally constructed by Gillet and Soul\'e, see e.g.\ \cite{SouleBook}.    A  discussion extending the construction to the case of Deligne-Mumford stacks of dimension two can be found in \cite[\S 2]{KRYbook}.
	
	Roughly, a class in $\ChowHat{n}_{\bbR}(\calX)$ is represented by an \emph{arithmetic cycle} $(\calZ, g)$, where $\calZ$ is a codimension $n$ cycle on $\calX$, with $\bbR$-coefficients, and $g_Z$ is a Green current for $\calZ(\bbC)$, i.e.\ $g_Z$ is a current on $\calX(\bbC)$ of degree $(n-1,n-1)$ for which there exists a smooth form $\omega$ such that
	\begin{equation} \label{eqn: green eqn general}
	\ddc g_Z + \delta_{\calZ(\bbC)} = [\omega],
	\end{equation}
	holds; here $[\omega]$ is the current defined by integration against $\omega$. The rational arithmetic cycles are those of the form $\widehat{\mathrm{div}}(f) = (\mathrm{div}f, i_*[ - \log |f | ^2])$, where $f \in \bbQ(W)^{\times}$ for a codimension $n-1$ integral subscheme $\iota \colon W \hookrightarrow \calX$, together with classes of the form $(0, \partial \eta + \overline{\partial} \eta')$; by definition, the arithmetic Chow group $\ChowHat{n}_{\bbR}(\calX)$ is the quotient of the space of arithmetic cycles by the $\bbR$-subspace spanned by the rational cycles.
	
	Note that if $n=2$, for dimension reasons $\calZ(\bbC) =0$ and  Green's equation \eqref{eqn: green eqn general} is devoid of content; in particular, there is nothing linking $\calZ $ and $g$ in this case.
	
	We will also require a generalization of these Chow groups due to K\"uhn and Burgos-Kramer-K\"uhn. For a cusp $P$ of $X_0(N)$, let $\calP$ denote its flat closure in $\calX_0(N)$, and let $\calS = \sum \calP$ denote the cuspidal divisor.  In \cite{BKK}, the authors  present an abstract  framework that allows for more general kinds of Green objects in the the theory of arithmetic Chow groups. One of the examples they present utilizes ``pre-log-log forms" relative to a fixed normal crossing divisor. In our case, we take this fixed divisor to be the cuspidal divisor $\calS$, and  denote by
	\begin{equation}
	\ChowHat{\bullet}_{\bbR}(\calX, \Dpre ) = \bigoplus_n \ChowHat{n}_{\bbR}(\calX, \Dpre)
	\end{equation}
	the corresponding Chow ring. In particular, there is a natural map
	\begin{equation}
	\ChowHat{n}(\calX) \to \ChowHat{n}(\calX, \Dpre),
	\end{equation}
	an intersection product
	\begin{equation}
	\ChowHat{1}_{\bbR} (\calX, \Dpre) \times \ChowHat{1}_{\bbR}(\calX, \Dpre) \to \ChowHat{2}_{\bbR}(\calX, \Dpre) \qquad (\widehat Z_1,\widehat Z_2) \mapsto \widehat Z_1 \cdot \widehat Z_2,
	\end{equation}
	and a degree map
	\begin{equation}
	\widehat {\mathrm{deg} } \colon \ChowHat{2}_{\bbR}(\calX, \Dpre) \to \bbR.
	\end{equation}
	For convenience, we will often abbreviate
	\begin{equation}
	\langle \widehat \calZ_1, \widehat  \calZ_2 \rangle = \widehat{\mathrm{deg}} \left( \widehat \calZ_1 \cdot \widehat \calZ_2 \right) , \qquad \widehat \calZ_1, \widehat \calZ_2 \in \ChowHat{1}(\calX, \calD_{\mathrm{pre}})
	\end{equation}
	for the intersection pairing, which recovers the intersection pairing defined in \cite{Kuhn}.
	We will not require precise formulas for any of the aforementioned structures; the interested reader may consult \cite[\S 7]{BKK} for a complete discussion.
	
	\begin{remark} \label{rmk:stacky 1/2} Strictly speaking, the constructions in \cite{BKK} and \cite{Kuhn} only apply to schemes, and not Deligne-Mumford stacks. In our case, this technicality is essentially immaterial. Indeed, we have identifications
		\begin{equation}
		\calX(\bbC) \simeq \left[ \Gamma_0(N) \backslash \bbH \right] \simeq \left[ \{ \pm 1\} \backslash X_0(N) \right]
		\end{equation}
		as orbifolds, where $X_0(N)$ is the Riemann surface $\Gamma_0(N) \backslash \bbH$, and $\{ \pm 1\}$ acts trivially on $X_0(N)$. Hence the analytic constructions of \cite{BKK} can be carried out on $X_0(N)$; the only difference is that in numerical formulas, we include a factor of $1/2$ to reflect the presence of the group $\{ \pm 1\}$.
	\end{remark}
	\subsection{Special divisors}
	Here we recall the  special divisors defined in \cite{DuYang}, which include certain additional boundary components. For $v \in \bbR_{>0}$ and $m \in \bbZ$,  let
	\begin{equation}
	g(m,v) = \begin{cases}
	\frac{\sqrt{N}}{2 \pi \sqrt{v}} \beta_{3/2}(- 4 \pi m v), & \text{if } m <0 \text{ and } -Nm \text{ is a square} \\[10pt]
	\frac{\sqrt{N}}{2 \pi \sqrt{v}}, & \text{if } m = 0 \\[10pt]
	0, & \text{otherwise},					
	\end{cases}
	\end{equation}
	where, for $s \in \bbR$,
	\begin{equation}
	\beta_{s}(r) = \int_1^{\infty}	e^{-rt} t^{-s} dt.
	\end{equation}	
	The modified special cycle is defined to be		
	\begin{equation} \label{def:modified divisor}
	\calZ^*(m,v) := \calZ(m)  + g(m,v) \calS.
	\end{equation}
	where $\calS  = \sum\limits_{\calP  \, \mathrm{ cusps} } \calP$ is the cuspidal divisor.
	
	We now describe Green functions for these cycles, following \cite{KudlaCentralDerivs}. First, for $x \in \calV$ and $z \in \bbD$, let
	\begin{equation}
	R(x,z) := -\frac{ | \langle x, \zeta \rangle|^2}{\langle \zeta, \overline \zeta \rangle}
	\end{equation}
	where $\zeta \in z$ is any non-zero vector; note that the definition is clearly independent of $\zeta$. For $v \in \bbR_{>0}$,  and $m \neq 0$, we define \emph{Kudla's Green function}
	\begin{equation}
	\xi(m,v) := \sum_{\substack{x \in \calL\\ Q(x) = m}} \int_1^{\infty} e^{- 2 \pi R(x,z)v t} \frac{dt}{t} ,
	\end{equation}
	which is a smooth function on $\calY_0(N)(\bbC) - \calZ(m)(\bbC)$. The  behaviour of these functions near the cusps was determined by Du and Yang. In particular, \cite[Theorem 5.1]{DuYang} implies that $\xi(m,v)$ is a Green function for $\calZ^*(m,v)$, and so we obtain a class
	\begin{equation}
	\widehat \calZ(m,v) \ := \left(\calZ^*(m,v) , \, \xi(m,v) \right) \in \ChowHat{1}_{\bbR}(\calX)
	\end{equation}
	
	The case $m=0$ requires a bit more notation. First, let $\varphi_{\mathrm{un}}\colon \calE_{\mathrm{un}} \to \calE'_{\mathrm{un}}$, denote the universal diagram over $\calX  = \calX_0(N)$, where $\pi \colon \calE_{\mathrm{un}} \to \calX$ is a generalized elliptic curve, and let
	\begin{equation}
	\omega_N = \pi_*\left( \Omega^1_{\calE_{\mathrm{un}}/\calX} \right)
	\end{equation}
	denote the Hodge bundle; sections of this bundle correspond to modular forms of weight one. This bundle can be metrized  on $Y_0(N) =\calY_0(N)(\bbC)$ as follows: at a point $\tau \in \bbH^{\pm} \simeq \bbD$, corresponding to the diagram $\varphi_{\tau} \colon E_{\tau} \to E'_{\tau}$ as in  \eqref{eqn:complex moduli point}, the metric $\| \cdot \|_{\tau}$ at $\tau$ is determined by the formula
	\begin{equation} \label{eqn:metric}
	\| dz \|_{\tau}^2 =  (2 \sqrt{\pi} e^{-\gamma/2}) \, v, \qquad \tau = u + iv
	\end{equation}
	where $dz$ is the coordinate differential with respect to the uniformization $E_{\tau} = \bbC / (\bbZ + \bbZ \tau)$, and $\gamma = -\Gamma'(1)$ is the Euler-Mascheroni constant.  This metric is logarithmically singular at the cusps; this metrized line bundle induces a class
	in the generalized arithmetic Chow group, which  we also denote by
	\begin{equation}
	\widehat \omega_N \in \ChowHat{1}(\calX, \Dpre).
	\end{equation}
	
	Following \cite{DuYang}, we will need a slight modification to this bundle. Recall that for a prime $p |N$, the fibre $\calX_{p}$ at $p$ decomposes into two irreducible components $\calX_p^0$ and $\calX_p^{\infty}$, where $\calX_p^0$ (resp.\ $\calX_p^{\infty}$) is the component that contains the reduction of the cusp $\calP_0$ (resp. $\calP_{\infty}$) mod $p$.  We then define
	\begin{equation} \label{eqn:widehat L def}
	\widehat \omega := - 2 \widehat \omega_N  - \sum_{p|N} \widehat \calX_p^0  + ( 0, \log N),
	\end{equation}
	where $\widehat \calX_p^0 = ( \calX_p^0, 0) \in \ChowHat{1}(\calX)$ is the corresponding class.
	\begin{remark}
		The class $\widehat \omega$ can be identified as (the class of) the dual of the Hodge bundle of the product $\calE_{un} \times \calE_{un}'$ over $\calX_0(N)$ with an appropriately scaled metric; see \cite[\S 2.2]{HowardUnitaryCurves} for details.
	\end{remark}

	Next, define
	\begin{equation}
	\xi(0,v) \ = \ \sum_{\substack{x \in \calL \\ Q(x) = 0 \\ x \neq 0}} \int_1^{\infty} e^{ - 2 \pi v R(x,z) t} \frac{dt}t
	\end{equation}
	which is a smooth function on $Y_0(N)$. As shown in \cite{DuYang}, it is a Green function for the divisor $g(0,v) S$, with $\log-\log$ singularities at the boundaries, which therefore determines a class
	\begin{equation} ( g(0,v) \calS, \xi(0,v)) \in \ChowHat{1}_{\bbR}(\calX, \Dpre).
	\end{equation}
	It turns out that the linear combination $\widehat{\omega} + ( g(0,v) \calS, \xi(0,v))$, in fact lands in $\ChowHat{1}_{\bbR}(\calX)$, the usual Gillet-Soul\'e Chow group, see \cite[Proposition 6.6]{DuYang}.
	
	Finally, we set
	\begin{equation}
	\widehat{\calZ} (0,v) \ := \ \widehat \omega + \left( g(0,v) \calS, \xi(0,v)
	\right) - (0, \log v ) \ \in \ \ChowHat{1}_{\bbR}(\calX)
	\end{equation}

	\subsection{Special cycles in codimension two} In this section, we will attach an arithmetic cycle $\widehat \calZ(T,v)$ for every $T \in \Sym_2(\bbZ)$ and $v \in \Sym_2(\bbR)_{>0}$.  \label{sect2.5}
	
	We begin with the case that $T$ is positive definite. As in \Cref{lemma:support of Z(T)}, if the cycle $\calZ(T)$ is non-zero, it is supported in the fibre $\calX_p$ for a prime $p$. We may then define
	\begin{equation}
	\widehat \calZ(T,v) \ := \ (\calZ(T), 0) \ \in \ \ChowHat{2}_{\bbR}(\calX);
	\end{equation}
	note   that this class is independent of $v$.
	
	Next, suppose $T$ is non-degenerate and of signature $(1,1)$ or $(0,2)$, so that $\calZ(T) = \emptyset$. In this case, we define a purely archimedean class via the construction of \cite{GarciaSankaran}. We briefly recall the construction in the case at hand.
	Suppose $x = (x_1, x_2) \in \calV_{\bbR}^{\oplus 2}$ is a linearly independent pair of vectors, and consider the Schwartz form
	\begin{equation}
	\nu(x)  \in S(\calV_{\bbR}^2) \otimes A^{1,1}(\bbH)
	\end{equation}
	defined in \cite{GarciaSankaran}. Explicitly, we have
	\begin{equation}
	\nu (x) \ := \ \nu^o(x) e^{- 2 \pi (Q(x_1 )+ Q(x_2))}
	\end{equation}
	where
	\begin{equation}
	\nu^o(x) = \psi^o(x) \, d \mu
	\end{equation}
	with
	\begin{equation} \label{eq:psi0}
	\psi^o(x,\tau) =  \left(- \frac{1}{\pi} +  2 \sum_{i =1}^2( R(x_i,\tau) + 2 Q(x_i)) \right) e^{-2\pi (R(x_1, \tau) + R(x_2, \tau))}, \qquad x= (x_1, x_2)
	\end{equation}
	and
	\begin{equation}
	d \mu = \frac{du \wedge dv}{v^2}
	\end{equation}
	for $\tau = u + i v \in \bbH$; here we are abusing notation and writing $R(x, \tau) = R(x, [z_{\tau}])$ under the identification $\bbH^{\pm} \simeq \bbD$ as in \eqref{eqn:upper half plane sym space isom}.
	
	Next, we set
	\begin{equation}
	\Psi(x,\tau) = \int_1^{\infty}  \psi^o( \sqrt{t} x,\tau) \, \frac{dt}{t}.
	\end{equation}

	If the components of $x$ span a space of signature $(1,1)$ or $(0,2)$, then $\Psi(x, \cdot )$ is a smooth  function on $\mathbb H$.
	
	Finally, for $T \in \Sym_2(\bbQ)$ of signature $(1,1)$ or $(0,2)$, and $v \in \Sym_2(\bbR)_{>0}$, we choose  $a \in \GL_2(\bbR)$ with $v = a \cdot {}^t a$ and set
	\begin{equation}
	\Psi(T,v,\tau) = \sum_{\substack{ x \in \calL^2 \\ T(x) = T}} \Psi(xa,\tau)
	\end{equation}	
	and
	\begin{equation} \label{eqn:Green current indefinite case}
	\Xi(T,v) :=  \Psi(T,v,\tau) d\mu(\tau) .
	\end{equation}
	The sum converges absolutely to a  $\Gamma_0(N)$-invariant function on $\bbH$, and $\Xi(T,v)$ defines a smooth form on $Y_0(N)$.
	
	\begin{lemma} The form $\Xi(T,v)$ is absolutely integrable on $X_0(N)$, i.e.\
		\[
		\int_{X_0(N)} | \Psi(T,v,\tau) | d\mu(\tau) < \infty.
		\]
		\begin{proof} We first claim that it suffices to prove that for any $x = (x_1, x_2) \in \Omega(T)$, the form $\Psi(xa, \tau) d\mu(\tau)$ is absolutely integrable on $\mathbb H$; indeed, if this is the case, we would have
			\begin{equation} \label{eq:Psi}
			\int_{X_0(N)} \Psi(T,v,\tau) d \mu(\tau) = \int_{\Gamma_0(N)\backslash \bbH} \left( \sum_{\substack{ x \in L^2 \\ T(x) = T}}  \Psi(xa, \tau) \right) d\mu(\tau) = \sum_{\substack{x \in \Omega(T) \\ \text{mod} \ \Gamma_0(N)} }  \int_{\bbH} \Psi(xa, \tau) \, d \mu(\tau),
			\end{equation}
			where we use the fact that the stabilizer of $x$ in $\Gamma_0(N)$ is $\{\pm 1\}$, which acts trivially on $\bbH$; since there are finitely many $\Gamma_0(N)$ orbits, the lemma would follow by Fubini.
			
			To show the integrability of $\Psi(xa, \tau)$, fix a matrix $k \in SO(2)$ such that
			\begin{equation}
			x a k = (x_1,x_2) a k = (y_1, y_2)
			\end{equation}
			for a tuple $y = (y_1, y_2) \in V_{\bbR}^2$ such that
			\begin{equation}
			T(y) = \begin{pmatrix} \delta_1 & \\ & \delta_2 \end{pmatrix},
			\end{equation}
			with $\delta_1, \delta_2 \in \bbR^{\times}$ and $\delta_1 < 0$. It follows from \cite{GarciaSankaran} that $\Psi(xa, \tau) = \Psi(y, \tau)$. Furthermore, we have
			\begin{equation}
			\Psi(\gamma y, \tau) = \Psi(y, \gamma \cdot \tau)
			\end{equation}
			for $\gamma \in \GL_2(\bbR)$, acting diagonally on $V_{\bbR}^2$. Hence, in the integral over $\bbH$, we may act by an appropriate choice of $\gamma$ and assume without loss of generality that
			\begin{equation}
			y_1 = |\delta_1|^{\frac12} \begin{pmatrix} 1 & \\ & -1 \end{pmatrix}, \qquad y_2 =|\delta_2|^{\frac12} \begin{pmatrix} & 1 \\ \pm 1 & \end{pmatrix}
			\end{equation}
			where the sign in the matrix defining $y_2$ is $-$ if $T$ is of signature $(1,1)$, and $+$ if $T$ is of signature $(0,2)$.
			
			First, suppose $T$ has signature $(0,2)$.
			In this case, we have
			\begin{equation}
			R_1 \ := \ R(y_1, \tau) = \frac{ |\delta_1|}{4 v^2}  \cdot |\tau|^2, \qquad R_2 := R(y_2, \tau) = \frac{|\delta_2|}{4 v^2}\cdot |1 - \tau^2 |^2.
			\end{equation}
			Setting $c = \frac{\pi}2 \min(|\delta_1|, |\delta_2|)$ and writing $\tau = u + iv$, we then have
			\begin{align*}
			2 \pi(	R_1 + R_2) &\geq \frac{c}{v^2} \left( |\tau|^2 + |1 - \tau^2|^2 \right) \\
			& = \frac{c}{v^2} \left( u^2 + v^2 + (1- u^2)^2 + 2v^2(1+u^2) + v^4 \right)  \\
			&= \frac{c}{v^2} \left( v^2 + (1/2 - u^2)^2 + 3/4 + 2v^2(1+u^2) + v^4 \right) \\
			&> c \left( 3/4 \cdot v^{-2} + v^2 + 2 u^2 \right).
			%	&> \frac34 \cdot \frac{ c}{v^2}
			\end{align*}
			This estimate easily implies that the integrals
			\begin{equation}
			\int_{\bbH} \int_1^{\infty} e^{-2 \pi t (R_1+R_2)}  \frac{dt}{t} d\mu(\tau) < \int_{\bbH} \frac{1}{2 \pi (R_1 + R_2)} e^{-2 \pi (R_1 + R_2)} d\mu(\tau)
			\end{equation}
			and
			\begin{equation}
			\int_{\bbH} \int_1^{\infty} (t R_i) \cdot e^{-2 \pi t (R_1+R_2)} \frac{dt}{t}d\mu(\tau) = \int_{\bbH} \frac{R_i}{2\pi (R_1 + R_2)}  e^{-2 \pi (R_1+R_2)} d \mu(\tau)
			\end{equation}
			are finite. So the integral (\ref{eq:Psi}) is absolutely  convergent,  which in turn implies the statement of the lemma. The case of signature $(1,1)$ is analogous.
		\end{proof}
	\end{lemma}
	The preceding lemma implies that $\Xi(T,v)$ defines a current on $X_0(N)$ so in this case we obtain a class
	\begin{equation}
	\widehat \calZ(T,v) = 	(0, \Xi(T,v) ) \in \ChowHat{2}_{\bbR}(\calX).
	\end{equation}
	
	Next, we suppose $T$ has rank 1.
	We begin with the following simple observation:
	
	\begin{lemma} \label{lem:rank 1 T}
		Suppose $\mathrm{rank}(T) =1$. Then there exists $t \in \mathbb Z$ and an element $\gamma \in \GL_2(\bbZ)$ such that
		\begin{equation} T =  {}^t\gamma \,  \psm{0 & \\ & t }\, \gamma  			.	
		\end{equation}
		
		Moreover, the integer $t$ is uniquely determined by $T$, and is invariant upon replacing $T$ by ${}^t \sigma T \sigma$ with $\sigma \in \GL_2(\bbZ)$.
		\begin{proof}
			For the existence, write $T = \psm{t_1 & m \\ m & t_2}$, so that $m^2 = t_1 t_2$. If $t_1 = 0$, then we may take $\gamma = Id $ and $ t= t_2$. Otherwise, write $ - m / t_1 = a/c$ where $a $ and $c$ are relatively prime, and choose $b,d$ such that $ad - bc = 1$; then $\gamma = \psm{a & b\\ c & d}$ satisfies ${}^t \gamma T \gamma = \psm{0 & \\ & t}$ for some $t$, as required.

			To show that $t$ is uniquely determined, suppose that we have $\gamma \in \GL_2(\bbZ)$ with
			\begin{equation}
			{}^t\gamma \,  \psm{0 & \\ & t }\, \gamma =    \psm{0 & \\ & t' }
			\end{equation}
			Then $\gamma $ is of the form $\psm{ a & b \\ 0 & d}\in \GL_2(\bbZ)$, hence $d = \pm 1$, which in turn implies that $t' = t$.
		\end{proof}
	\end{lemma}
	
	Following \cite{KRYbook}, we make the following definition.
	\begin{definition}  \label{def:Z(T) rank 1}
		
		Suppose $T$ is of rank 1, and let $v \in \Sym_2(\bbR)_{>0}$. Let $t$ be as in \Cref{lem:rank 1 T}.
		Define
		\begin{equation}
		\widehat{\calZ}(T,v)  \ := \   \widehat{\calZ}(t,v_0)  \cdot \widehat{\omega}  - \left( 0, \log \left( \frac{\det v}{v_0}\right) \delta_{\calZ^*(t, v_0)(\bbC)} \right)  \in \ChowHat{2}(\calX,\Dpre),
		\end{equation}
		where
		\begin{equation}
		v_0 :=  t^{-1} \mathrm{tr}(Tv) \in \mathbb R_{>0},
		\end{equation}		
		the cycle $\calZ^*(t, v_0)(\bbC)$ is the complex points of the modified cycle \eqref{def:modified divisor},
		and $\widehat \omega$ is the class defined in \eqref{eqn:widehat L def}
	\end{definition}
	
	We observe two invariance properties, which both follow immediately from definitions: the first is the identity
	\begin{equation} \label{eqn:calZ SL2 invariance}
	\widehat\calZ( {}^t \gamma T \gamma, v) = \widehat \calZ(T, \gamma v {}^t \gamma), \qquad \gamma \in \GL_2(\bbZ).
	\end{equation}
	For the second invariance property, suppose $\theta = \psm{1 & * \\ & 1} \in \mathrm{SL}_2(\bbR)$; then
	\begin{equation} \label{eqn:calZ theta invariance}
	\widehat \calZ\left( \psm{0 & \\ & t}, \theta v {}^t \theta \right) = \widehat\calZ \left( \psm{0&\\ & t}, v \right)
	\end{equation}
	for any $t \in \bbZ$ and $v \in \Sym_{2}(\bbR)_{>0}$.

	It remains to define the constant term.
	\begin{definition}
		When $T=0$, we set
		\begin{equation}
		\widehat \calZ(0,v) =  \widehat \omega \cdot \widehat \omega  \ +  \ \left( 0,
		\log \det v  \cdot [\Omega] \right)  \ \in \ \ChowHat{2}(\calX ,\Dpre),
		\end{equation}
		where $\Omega = \frac{dx \wedge dy}{2 \pi y^2}$, and $[\Omega]$ is the  current given by integration against $\Omega$.
	\end{definition}

	\subsection{The main theorem}  \label{sec:main theorem setup}

	We begin by briefly reviewing the theory of Siegel Eisenstein series, mostly to fix notation; for details, see \cite[\S 8]{KRYbook}, for example.  Let
	\begin{equation}
	\Sp_r = \left\{ g \in \GL_{2r} \ | \ {}^tg  \begin{pmatrix} & 1_r \\ -1_r &  \end{pmatrix} g =   \begin{pmatrix} & 1_r \\ -1_r &  \end{pmatrix} \right\}
	\end{equation}
	viewed as an algebraic group over $\bbQ$.
	
	For a place $v \leq \infty$, let $G_{r,v}$ denote the metaplectic cover of $\Sp_{r}(\bbQ_v)$; as a set,  we have $G_{r,v} = \Sp_r(\bbQ_v) \times \bbC^{1}$   with multiplication given by the normalized Leray cocycle (see \cite[\S 8.5]{KRYbook}).
	Similarly, $G_{r, \bbA} = \Mp_r(\bbA)$ denote the metaplectic cover of $\Sp_r(\bbA)$, which can be realized as a restricted direct product of the $G_{r,v}$. There is a unique splitting $\Sp_r(\bbQ) \to G_{r, \bbA}$ whose image we denote by $G_{r,\bbQ}$.
	Let $P_r = M_rN_r$ denote the standard Siegel parabolic of $\Sp_r$; here
	\begin{equation} \label{eqn:Siegel parabolic defns}
	M_r = \left\{ m(a) = \begin{pmatrix}a & \\ & {}^t a^{-1} \end{pmatrix} \ \Big| \ a \in \GL_r\right\} \qquad \text{and} \qquad 	N_r= \left\{ n(b) = \begin{pmatrix}1 &b \\ & 1 \end{pmatrix} \ \Big| \ b \in \Sym_r \right\} .
	\end{equation}
	We denote by $\tilde P_{r,\bbA} , \tilde M_{r, \bbA}$ and $\tilde N_{r, \bbA}$, the inverse images of $P_r(\bbA)$, $M_r(\bbA)$, and $N_r(\bbA)$ under the covering map $G_{r,\bbA} \to \Sp_r(\bbA)$.

	Given the quadratic space $\calV$ as in \eqref{eqn:calV defn}, we may define a character $\chi_{\calV}$ on $\tilde P_{\bbA}$ by the formula
	\begin{equation}
	\chi_{\calV} \left([m(a) n(b), \epsilon] \right)) = \epsilon \left( \det(a), - \det \calV \right)_{\bbA} = \epsilon \left( \det a , - N \right)_{\bbA} ,
	\end{equation}
	where $(\cdot, \cdot )_{\bbA}$ is the global Hilbert symbol. Similarly, for $s \in \mathbb C$, we define a character $| \cdot | ^s$ on $\widetilde P_{\bbA}$ by
	\begin{equation}
		\left| p \right|^{s} = | \det (a)|^s_{\bbA}, \qquad p = [m(a) n(b), \epsilon].
	\end{equation}
	Consider the degenerate principal series representation
	\begin{equation}
	I_r(s,\chi_{\calV}) := \{ \Phi\colon G_{r,\bbA} \to \bbC \text{ smooth} \ | \ \Phi(pg, s) = \chi_{\calV}(p) |p|^{s+\rho}  \Phi(g,s) \text{ for all } g \in  G_{r, \bbA}, p \in \tilde P_{r,\bbA}\},
	\end{equation}
	with $G_{r,\bbA}$ acting by right translation; here $\rho = \frac{r+1}2$.
	This representation factors as a restriced direct product
	\begin{equation}
	I_r(s,\chi_{\calV}) = {\bigotimes_{v \leq \infty}}' I_{r,v}(s,\chi_{\calV,v}).
	\end{equation}
	where
	$$
	I_{r,v}(s,\chi_{\calV,v}) := \{ \Phi\colon G_{r,v} \to \bbC \text{ smooth} \ | \ \Phi(pg, s) = \chi_{\calV,v}(p) |p|_v^{s+\rho}  \Phi(g,s) \text{ for all } g \in G_{r,v}, p \in \tilde P_{r,v}\},
	$$
	is the local analogue of $I_r(s, \chi_{\calV})$.
	
	For each finite prime $p$, let  $ K_p  = \Sp_{r}(\bbZ_p)$ and let $\tilde K_p$ denote its inverse image in $ G_{r,p}$. Similarly, let $K_{\infty} \simeq U(r)$ denote the standard maximal compact subgroup of $\Sp_{r}(\bbR)$, let $\tilde K_{\infty}$ denote its inverse image in $G_{r,\infty}$, and set $\tilde K \subset G_{r, {\bbA}} $  to be the inverse image of $K_{\infty} \prod_p K_p$.
	
	We say that a section $\Phi \in I_r(s, \chi_{\calV})$ (resp.\ $\Phi_v \in I_{r,v}(s, \chi_{\calV,v})$) is \emph{standard} if its restriction to $\tilde K$ (resp. $\tilde K_v$) is independent of $s$. Note that by the Cartan decomposition, a standard section is determined by its value at any fixed $s \in \bbC$.
	
	In this paper, we will primarily be concerned with following standard sections for $r=1,2$.
	
	\begin{definition} \label{def:calL global section}Let $\calL$ be the lattice described in \eqref{eqn:lattice def}. We define a standard section
		\begin{equation}
		\Phi_r^{\calL}(s) = \otimes_{v\leq \infty} \Phi^{\calL}_{r,v}(s) \in I_r(s, \chi_{\calV})
		\end{equation}
		as follows:
		\begin{itemize}
			\item 	At the place $\infty$, we set $\Phi^{\calL}_{r,\infty}(s)= \Phi_{r,\infty}^{3/2}(s)$, the standard section of scalar $\tilde K_{\infty}$-type $\det^{3/2}$. More precisely, there exists a character $\xi_{\frac12}$ of $\tilde K_{\infty}$ whose square descends to the determinant map on $U(r)$, cf. \cite[\S 8.5.6]{KRYbook}. Then $\Phi^{\calL}_{r,\infty}(s)$ is the standard section determined by the relation $\Phi^{\calL}_{r,\infty}(k, s) = \xi_{\frac12}(k)^3$ for all $k \in \tilde K_{\infty}$.
			\item  For $v < \infty$,  let $\Phi^{\calL}_{r,v}(g,s)$  be the local standard section attached to the lattice $\calL_{v}^{\oplus r}$ via the Rallis map; more precisely, $\Phi_{r,v}^{\calL}$ is determined by the relation
			\begin{equation}
			\Phi_{r,v}^{\calL}(g,s_0) =  \left( \omega_{\calV,v}(g) \varphi_v^{\calL} \right) (0), \qquad s_0  = \frac{3-r-1}2
			\end{equation}
			where $\omega_{\calV, v}$ is the Weil representation acting on $S(\calV_v^r)$, and $\varphi_v^{\calL}$ is the characteristic function of $\calL_v^r$.
		\end{itemize}	
	\end{definition}
	
	Our primary interest is in the case $r=2$, though we will have some need to consider the case $r=1$ as well.
	For $Re(s) > 3/2$ and $g \in G_{2,\bbA}$, let
	\begin{equation}
	E(g, s, \Phi^{\calL}_2) \ = \ \sum_{\gamma \in P_{2,\bbQ} \backslash G_{2,\bbQ} }  \Phi_2^{\calL}(\gamma g, s)
	\end{equation}
	denote the corresponding Eisenstein series; by the general theory of Eisenstein series, this series admits a meromorphic continuation to $s \in \bbC$. Moreover, $E(g,s, \Phi_2^{\calL})$ is \emph{incoherent} in the sense of \cite{KudlaCentralDerivs}, and hence $E(g, 0, \Phi_2^{\calL}) = 0$.
	
	It will be convenient to introduce ``classical" coordinates as follows. Let
	\begin{equation}
	\bbH_2 = \left\{ \tau  =u + iv \in \Sym_{2}(\bbC) \ | \ v > 0 \right\}
	\end{equation}
	denote the Siegel upper half-space; for $\tau = u + iv \in \bbH_2$,  choose any matrix $a \in \GL_2(\bbR)$ with $\det(a) > 0$ and $v = a \cdot {}^t a$, and write
	\begin{equation} \label{eqn:g_tau def}
	g_{\tau, \infty}= [n(u) m(a), 1]  \in G_{2,\infty}, \qquad g_{\tau} = \left( g_{\tau, \infty}, 1, \dots \right) \in G_{2,\bbA}.
	\end{equation}
	Define
	\begin{equation}
	E(\tau, s, \Phi_2^{\calL}) = \det(v)^{-3/4} E(g_{\tau}, s, \Phi_2^{\calL}),
	\end{equation}
	which is a (non-holomorphic) Siegel modular form of scalar weight $3/2$. We write its  $q$-expansion as
	\begin{equation}
	E(\tau, s, \Phi_2^{\calL}) = \sum_{T \in \Sym_2(\bbQ)} C_T(v, s, \Phi_2^{\calL}) \, q^T
	\end{equation}
	with $q^T = e^{2 \pi i tr(Tv)}$.
	
	\begin{theorem}
		For every $T \in \Sym_2(\bbQ)$, we have
		\begin{equation}
		\widehat \deg \, \widehat \calZ(T,v)  = \left( \frac{\prod_{p|N}(p+1)}{24} \right)\frac{d}{ds} C_T(v,s,\Phi_2^{\calL}) \Big|_{s=0}.
		\end{equation}
		In particular, we have an identity of $q$-expansions
		\begin{equation}
		\sum_T \widehat \deg \, \widehat \calZ(T,v) q^T = \left( \frac{\prod_{p|N}(p+1)}{24} \right) E'(\tau, 0, \Phi_2^{\calL}).
		\end{equation}
	\end{theorem}
	This theorem will be proved in \Cref{sec:proof of main thm}.

	\section{Local special cycles and Whittaker functionals} 	\label{sec:local}
	In this section, we study local analogues of the special cycles, defined in terms of deformations of $p$-divisible groups. A result of Gross-Keating computes the degrees of these cycles, which we relate to derivatives of Whittaker functionals.

	\subsection{Degrees of local special cycles}

	Fix a prime $p$,  let $\bbF = \overline{\bbF}_{p}$ be an algebraic closure of $\bbF_p$, $W = W(\bbF)$ the ring of Witt vectors and $W_{\bbQ} = W \otimes_{\bbZ} \bbQ$ its field of fractions. Denote by $\Nilp_{p}$ the category of local $W$-algebras such that $p$ is nilpotent.
	
	Let $\bbX$ denote the (unique, up to isomorphism) supersingular $p$-divisible group of height 2 and dimension 1 over $\bbF$. Then $\End(\bbX)$ is the maximal order in the division  quaternion  algebra over $\bbQ_{p}$; we denote the main involution by $ x \mapsto x^{\iota}$, and the reduced norm by $\nm(x)$.
	
	Fix a uniformizer $\varpi$. Let $\varphi \in \End(\bbX)$ be an isogeny of degree $N$. By composing with an automorphism of $\bbX$, we may assume without loss of generality that
	\begin{equation}
	\varphi = \begin{cases} id, & p \nmid N\\
	\varpi, & p|  N.
	\end{cases}
	\end{equation}
	
	With this setup in place, we recall the relevant Rapoport-Zink space: let $\calM = \calM_{\Gamma_0(N)}$ denote the moduli space (over $\Nilp_{p}$) of diagrams $(\tilde\varphi\colon X \to X')$ where $X$ and $X'$ are deformations of $\bbX$ and $\tilde \varphi$ is an isogeny lifting $\varphi \in \End(\bbX)$.
	
	If $p \nmid N$, then $\varphi = id$, and $\calM_{\Gamma_0(N)}$ is isomorphic to the moduli space $\calM_0$ of deformations of $\bbX$, which in turn is isomorphic to  $\Spf(W \llbracket t \rrbracket)$.
	
	\begin{definition}[Local special cycles]
		\begin{enumerate}[(i)]
			\item
			For  $ y\in \End(\bbX)^{tr=0}$, let $Z(y)$ denote the moduli problem (over $\Nilp_p$) parametrizing tuples $\{ (\phi \colon X_1 \to X_2, \delta)\}$ where
			\begin{itemize}
				\item $X_1$ and $X_2$ are deformations of $\bbX$;
				\item   $\phi$ is an isogeny lifting $\varphi$;
				\item $\delta \in \Hom(X_2, X_1)$ is an isogeny lifting $y \circ \varphi^{-1}$.
				%			\item $\alpha$ is an endomorphismsof $\phi$ lifting $y$;
				%			\item
				%			and $\alpha \cdot \ker \phi = 0$.
			\end{itemize}
			Note that if the last condition holds, then $\delta \circ \phi$ is an element of $\End(\phi)$ lifting $y$.
			%		Note that last condition is equivalent to the statement that the quasi-isogeny $\alpha \circ \phi^{-1}\colon X_2 \to X_1$ is a homomorphism.
			
			\item If $y = (y_1, y_2)$ is a linearly independent pair of elements in $\End(\bbX)^{tr=0}$, then we set
			\begin{equation}
			Z(y) = Z(y_1) \times_{\calM} Z(y_2)
			\end{equation}
			to be the intersection.
		\end{enumerate}
		
	\end{definition}
	
	For a pair of vectors $f_1, f_2 \in \End(\bbX)$, let $\langle f_1, f_2 \rangle = f_1 \cdot f_2 ^{\iota} + f_2 \cdot f_1^{\iota}$ denote the bilinear form attached to the quadratic form $Q(f) = \mathrm{nm}(x)$, where $\mathrm{nm}(x)$ is the reduced norm. We set
	\begin{equation} \label{eq:T}
	T(f_1, f_2) \ := \ \frac12 \begin{pmatrix} \langle f_1, f_1 \rangle & \langle f_1, f_2 \rangle \\ \langle f_1, f_2 \rangle & \langle f_2, f_2  \rangle \end{pmatrix}.
	\end{equation}
	
	\begin{proposition}[Gross-Keating] \label{prop:local deg}
		Suppose  $y_1, y_2  \in \End(\bbX)^{tr=0}$ is a linearly independent pair of vectors , and let  $\delta_i = y_i \circ \varphi^{-1}$ for $i = 1,2$. Let
		\begin{equation}
		T = T(\delta_1, \delta_2)
		\end{equation}
		and let $0 \leq a_1 \leq a_2 \leq a_3$ denote the Gross-Keating invariants of the matrix $\mathrm{diag}(N, T)$. Then
		\begin{equation}
		\nu_p(T) :=	\deg Z(y)
		\end{equation} only depends on $T$, and is given by the following explicit formulas.
		
		If $a_1 \equiv a_2 \pmod 2$, then
		\begin{align}
		\nu_p(T) &= \sum_{i=0}^{a_1 - 1} (i+1)(a_1 + a_2 + a_3 - 3i)p^i + \sum_{i = a_1}^{(a_1 + a_2 - 2)/2} (a_1 + 1) (2 a_1+a_2 + a_3 - 4i)p^i \\
		& \qquad \qquad + \frac{a_1+1}{2}(a_3 - a2 + 1) p^{(a_1 + a_2)/2}
		\end{align}
		and if $a_1 \not\equiv a_2 \pmod 2$, then
		\begin{equation}
		\nu_p(T) = \sum_{i=0}^{a_1 - 1} (i+1) (a_1 + a_2 + a_3 - 3i) p^i + \sum_{i = a_1}^{(a_1 + a_2 - 1)/2} (a_1 + 1)(2a_1 + a_2 + a_3 - 4i)p^i.
		\end{equation}
		\begin{proof}
			Given any $f \in \End(\bbX)$, let $Z(f)_{GK}$ denote the locus on $\calM_0 \times \calM_0$ on which $f$ deforms to an isogeny $f \colon X_1 \to X_2$ where $X_1$ and $X_2$ are deformations of $\mathbb X$. Then the deformation locus $Z(y)$ coincides with the triple intersection $Z(\varphi)_{GK} \cdot Z(\delta_1)_{GK} \cdot Z(\delta_2 )_{GK}$ on $\calM_0 \times \calM_0$.
			
			Note that for $i = 1,2$,
			\begin{equation}
			\langle \delta_i,  \varphi \rangle  = \delta_i \cdot \varphi^{\iota} +  \varphi \cdot (\delta_i)^{\iota}  =  (y_1 \varphi^{-1}) \cdot (- \varphi) + \varphi (- \varphi^{-1} \cdot y_i^{\iota}) =  - tr(y_i) = 0
			\end{equation}
			Therefore, the restriction of the quadratic form $Q(x) = \deg(x)$ to $\mathrm{span}\{  \varphi, \delta_1, \delta_2 \}$ is represented by the matrix $\mathrm{diag}(N, T)$ with respect to the basis $\{ \varphi, \delta_1, \delta_2 \}$. The desired formula is then \cite[Proposition 5.4]{GK}.
		\end{proof}
	\end{proposition}
	
	\subsection{Local Whittaker functionals and special cycles}
	
	The main point of this section is to express  the local intersection number, in the previous section, in terms of Whittaker functionals. Recall that for a standard section $\Phi_p \in I_{2,p}(s, \chi_{\calV,p})$ and $T \in \Sym_2(\bbQ_p)$, the Whittaker functional is defined to be
	\begin{equation}
	W_{T,p}(g, s, \Phi_p) := \int_{\Sym_2(\bbQ_p)}  \Phi( w_p^{-1}  n(b)g , s) \, \psi_p(- tr (Tb))  \, db
	\end{equation}
	where
	\begin{itemize}
		\item $w_p = \left[ \psm{& -1_2 \\ 1_2 & } , 1\right] \in G_p$;
		\item for $b \in \Sym_2(\bbQ_p)$, we write $n(b) = \left[ \psm{1 & b \\ & 1 }, 1 \right]$;
		\item $\psi_p \colon \bbQ_p \to \bbC$ is the standard additive character that is trivial on $\bbZ_p$
		\item and $db$ is the additive Haar measure on $\Sym_2(\bbQ_p)$ that is self-dual with respect to the pairing $(b_1 , b_2) \mapsto \psi_p(tr(b_1 b_2))$.
	\end{itemize}
	
	In addition to the section $\Phi_{p}^{\calL} = \Phi_{2,p}^{\calL}$ as in \Cref{def:calL global section}, we also need the following auxilliary section. Let
	\begin{equation}
	V^{ra}_{p} := (B_{p} )^{tr = 0}
	\end{equation}
	denote the space of traceless elements of the division quaternion algebra $B_{p}$ over $\bbQ_{p}$. We equip $V^{ra}_{p}$ with the quadratic form $Q^{ra}(x) = \hbox{nm}(x)$-reduced norm of $x$.  Note that the quadratic character $\chi'_p := \chi_{V^{ra}, p}$ associated to $V^{ra}_{p}$ is given by
	\begin{equation}
	\chi_{p}'(x) = (x, -1)_{p}.
	\end{equation}
	
	Let $L^{ra}_{p} = \calO_{p} \cap V^{ra}_{p}$, where $\calO_{p}$ is the maximal orde                                                                                                                  r; finally, we define the local section
	\begin{equation}
	\Phi^{ra}_{p}(s) \in I_{2,p}(s, \chi')
	\end{equation} to be the standard section attached to $(L^{ra}_{p})^{\oplus 2}$.
	The main result of this section is:
	\begin{proposition} \label{prop:local intersection number}
		Suppose  $y_1, y_2 \in \End(\bbX)^{tr=0}$ are linearly  independent vectors, and let $\delta_i = y_i \circ \varphi^{-1}$ and
		\begin{equation}
		T :=  T(\delta_1, \delta_2) = N^{-1} T(y_1, y_2)
		\end{equation}
		as in \eqref{eq:T}. Then
		\begin{equation} \label{eq:local}
		\deg \, Z(y)  \cdot \log p = \nu_p(T) \log p =  2  c_p \left( \frac{ \gamma_p(V^{ra}_p)}{ \gamma_p(\calV)} \right)^2  \frac{W'_{T,p}(e, 0, \Phi^{\calL}_{p})}{W_{NT, p}(e, 0, \Phi^{ra}_{p})}
		\end{equation}
		where
		\begin{equation}
		c_p =  \left( \frac{1}{p-1} \right) \times
		\begin{cases}
		p+1, & \text{ if }  p | N \\
		1,  & \text{ if }p \nmid N . \\
		\end{cases}
		\end{equation}
		and $\gamma_p(V_p^{ra})$ and $\gamma_p(\calV)$  are the local Weil indices, cf.\ e.g. \cite[\S 8.5.3]{KRYbook}.

	\end{proposition}
	
	\begin{remark}
		 	If $p$ is odd, it follows from \cite[Appendix A]{Rao} that	$\gamma_p(V_p^{ra})^2 = 1$ and $\gamma_p(\calV_p)^2 = (-1,p)_p$.

	\end{remark}

	\begin{proof}[Proof for $p \nmid N$]
		First, we observe the following general fact: for a lattice $L$ over $\mathbb Z_{\ell}$, let $\Phi_{\ell}^L \in I_{\ell}(s,\chi_{L})$ denote the section corresponding to (the characteristic function of) the lattice $L^{\oplus 2}$. Then, for $T \in \Sym_{2}(\bbZ_{\ell})$, we have
		\begin{equation} \label{eqn:Whittaker rep dens relation general}
		W_{T,\ell}(e, s, \Phi^L_{\ell}) = \gamma(V_{\ell})^2 \cdot [L^{\vee} : L]^{-1} \cdot |2|_{\ell}^{\frac12} \cdot \alpha_{\ell}(X, T, L)|_{X = p^{-s}}
		\end{equation}
		where $V =L_{\bbQ_{\ell}}$, and $L^{\vee}  $ is the dual lattice, cf.\ \cite[Lemma 5.7.1]{KRYbook}. Here $\alpha_{\ell}(L, T, X)$ is the representation density polynomial, as in \cite{YangLocal,YangLocal2}.
		
		Now we return our situation, and consider the case $p \nmid N$. In this case, the result is contained in \cite{YangLocal2}. Indeed,  note that if $0 \leq a_1 \leq a_2 \leq a_3$ are the Gross-Keating invariants of $\mathrm{diag}(N, NT)$, which are also the Gross-Keating invariants of $\mathrm{diag}(1, T)$,  then $a_1 = 0$, cf.\ \cite[Appendix B]{YangLocal2}.
		
		Since $N $ is a p-adic unit, carefully tracing through the contructions in \cite{YangLocal2} shows that
			\begin{equation}
				\alpha_{p}(X, T, \calL) = \alpha_p(X, NT, L^0_p)
			\end{equation}
		where
		$$
		L^0_p=\left\{ \begin{pmatrix}  {a}&{b} \\ {c} &{-a} \end{pmatrix}:\, a, b, c \in \Z_p\right\}, \qquad Q(x ) = \det(x).
		$$

		Thus,  \cite[Proposition 5.7]{YangLocal2} gives
		\begin{align}
		\alpha_p'(1, T, \calL_p)|_{X=1} &= -(1-p^{-2})
		\begin{cases} \sum_{0\le i \le \frac{a_2-1}2} (a_2+a_3-4i)p^i,  &\text{if }  a_2 \equiv 0 \pmod 2, \\
		\sum_{0\le i \le \frac{a_2}2-1} (a_2+a_3-4i)p^i - \frac{a_2-a_3+1}2p^{\frac{a_2}2}, &\text{if } a_2 \equiv 1 \pmod 2
		\end{cases}  \notag \\
		&= - (1-p^{-2}) \nu_p(T).
		\end{align}

		On the other hand, \cite[Proposition 5.7]{YangLocal2} gives
		\begin{equation} \label{eq:ram}
		\alpha_p(1, NT, L_p^{ra}) = 2 (p+1);
		\end{equation}
		combining these propositions with \eqref{eqn:Whittaker rep dens relation general} yields the proposition for $p \nmid N$.
		The case $p|N$ will be dealt at the end of this section after some preparation.
	\end{proof}

	In the remainder of the section, we suppose $p | N$; in particular, $p$ is odd. We have
	\begin{equation}
	\calL_p =  \left\{ \begin{pmatrix} a & p^{-1 } b \\ c & -a \end{pmatrix} \ \Big| \ a,b,c \in \bbZ_p \right\}, \qquad Q(x) = N \det(x).
	\end{equation}
	The Gram matrix of $\calL_p$, with respect to the  $\bbZ_p$-basis $\{\psm{& N^{-1} \\ -1 &}, \psm{& N^{-1} \\ 1 &},  \psm{1 & \\ & -1}  \}$ is
	\begin{equation}
	\calS \ := \ \begin{pmatrix}1 & & \\ & -1 & \\ & &  -N \end{pmatrix}
	\end{equation}
	In particular, $[\calL^{\vee}:\calL]^{-1}= p^{-1}$ and hence
	\begin{equation}
	W_{T,p}(e, s, \Phi^{\calL}_p) = \gamma(\calV_p)^2 \cdot p^{-1} \alpha_{\ell}(X, T, \calL)\big|_{X =p^{-s}}.
	\end{equation}
	
	Our first step is an  formula for the representation density $\alpha_p(X, T, \calL)$, using the explicit formulas in \cite{YangLocal}. Recall that for a general $T \in \Sym_2(\bbZ_p)$ and lattice $L$, Yang decomposes the representation density as
	\begin{equation}
	\alpha_p(X, T, L) = 1 + R_1(X, T, L) + R_2(X, T, L)
	\end{equation}
	for some explicit polynomials $R_1(X, T, L)$ and $R_2(X, T, L)$, defined in \cite[\S 7]{YangLocal}.
	
	In our case, we compute $\alpha_p(X, T, \calL)$ via comparison to the representation density $\alpha_p(X, T, L_0)$, where $L_0 = M_2(\Z_p)^{\mathrm{tr} =0}$ with quadratic form
	$Q(x) =\det x$.
	
	\begin{lemma} \label{lem:rep dens relation}
		Suppose $p|N$ and let  $T \in \Sym_2(\bbZ_p)$. Then
		\begin{equation}
		\alpha_p(X, T, \calL) = X^{-2} \left( p \alpha(X, NT, L_0) \ + (X-p) \left( R_1(X, NT, L_0) + \frac{X+p}{p} \right) \right),
		\end{equation}
		where $R_1(X, NT, L_0)$ is the polynomial defined in \cite[Theorem 7.1]{YangLocal}
		
		\begin{proof}[Sketch of proof] In  \cite[\S 7]{YangLocal}, the representation density is expressed in terms of  polynomials $R_1(X, T, L)$ and $R_2(X, T, L)$, which are further decomposed in terms of explicit polynomials $I_{i,j}(X, T, L)$ with $i = 1, 2$ and $j = 1, \dots, 8$.  Unwinding the definitions of these polynomials, one can verify explicitly that
			\begin{align*}
			I_{1,j}(X,T, \calS) & = X^{-1}  I_{1,j}(X, NT, S_0) \\
			I_{2,j}(X, T, \calS) &= p X^{-2} I_{2,j}(X, NT, S_0), \qquad j = 1, \dots 7 \\
			I_{2,8}(X, T, \calS) &= p X^{-2} I_{2,8}(X, NT, S_0) - 1.
			\end{align*}
			The lemma then  follows from \cite[Theorem 7.1]{YangLocal}.
			
		\end{proof}

	\end{lemma}
	
	Explicit formulas for $\alpha(X, NT, L_0)$ and $R_1(X, NT, L_0)$ are as follows:
	\begin{proposition}[{\cite{Kitaoka}, \cite[\S 8]{YangLocal}}]   \label{prop:rep dens split}
		Suppose $p|N$ and $T$ is $\GL_2(\bbZ_p)$-equivalent to $\psm{ \epsilon_1 p^{a} & \\ & \epsilon_2 p^b}$. Let $M=N/p \in \bbZ_p^{\times}$ and define
		\begin{equation*}
		v_0^+ = (-M \epsilon_1, p)_p  \qquad
		v_1^+ = \begin{cases}
		(-M \epsilon_2, p)_p, &  \text{if } b \text{ is odd} \\
		(-\epsilon_1 \epsilon_2, p)_p, & \text{if } b \text{ is even}.
		\end{cases}
		\end{equation*}
		\begin{enumerate}[(i)]
			\item If $a$ is odd, then
			$$
			\alpha_p(X, NT, L_0) = (1 - p^{-2}X^2) \left\{ \sum_{0 \leq k < \frac{a+1}{2}}  p^k (X^{2k} + (v_0^+ X)^{a + b+2 - 2k}) + p^{\frac{a+1}{2}} \sum_{a+1 \leq k \leq b+1}(v_0^+ X)^k\right\}
			$$
			and
			\begin{equation*}
			R_1(X, NT, L_0) = (1-p^{-2}) \sum_{0<k\leq\frac{a+1}{2}} p^k X^{2k} + p^{\frac{a+1}{2}} (1- p^{-1} v_0^+ X) \sum_{a+1<k\leq b+1} (v_0^+X)^k.
			\end{equation*}
			\item If $a$ is even, then
			\begin{equation*}
			\alpha_p(X, NT, L_0) = (1-p^{-2}X^2)\left\{ \sum_{0 \leq 0 \leq a/2} p^k(X^{2k} + v_1^+ X^{a+b+2-2k}) \right\}
			\end{equation*}
			and
			\begin{equation*}
			R_1(X,T, L_0) = -1 + (1 - p^{-1}X^2) \sum_{0 \leq k \leq a/2} p^k X^{2k} + v_1 p^{a/2} X^{b+2}.
			\end{equation*}	
		\end{enumerate}
		\qed
	\end{proposition}

	\begin{proposition} \label{prop:whittaker local degree identity bad primes}
		Suppose that  $p|N$ and $p \in \mathrm{Diff}(T, \calV)$, i.e.\ $T$ is not represented by $\calV_p$. Then
		\begin{equation*}
		\frac{p}{1-p} \cdot 	\alpha_p'(1, T, \calL) = \nu_p(T).	
		\end{equation*}
		Here $\nu_p(T) = \deg(Z(y))$ is the intersection multiplicity given explicitly in \Cref{prop:local deg}, for any tuple $y = (y_1, y_2)$ with $T(y_1, y_2) = NT$.
		\begin{proof}
			This follows from a straightforward, though  tedious, computation using \Cref{lem:rep dens relation} and \Cref{prop:rep dens split}.
		\end{proof}
	\end{proposition}
	
	We can now conclude the proof of \Cref{prop:local intersection number}:
	
	\begin{proof}[Proof of \Cref{prop:local intersection number} for $p|N$:]
		Using \eqref{eqn:Whittaker rep dens relation general} and \cite[Proposition 8.7]{YangLocal}, one has the formula
		\begin{equation}
		W_{NT,p}(e, 0, \Phi_p^{ra}) = \gamma_p(V_p^{ra})^2 \cdot p^{-2} \cdot 2(p+1) =2  p^{-2} (p+1).
		\end{equation}
		On the other hand, combining \eqref{eqn:Whittaker rep dens relation general} and \Cref{prop:whittaker local degree identity bad primes}, we have
		\begin{equation}
		W'_{T, p}(e, s, \Phi^{\calL}_p) = \gamma(\calV_p)^2 \cdot p^{-1} \cdot \alpha'(1, T, \calL) \cdot (-\log p) =   \gamma_p(\calV_p)^2 \left( \frac{p-1}{p^2}  \right) \nu_p(T) \cdot \log p.
		\end{equation}
		From this, the proposition follows.
	\end{proof}

	\section{Proof of the main theorem} \label{sec:proof of main thm}
	\subsection{Positive definite $T$}
	
	Suppose $T$ is positive definite.	Our strategy in this case closely mirrors that of \cite{KRYbook}. We begin by recalling the following well-known facts about the Fourier coefficients of Siegel Eisenstein series, see e.g.\ \cite[\S 1]{KudlaCentralDerivs}:
	\begin{proposition} \label{prop:Whittaker dichotomy}
		Suppose $T \in \Sym_r(\bbQ)$ is non-degenerate, that $V$ is a quadratic space over $\bbQ$ of dimension $r+1$, and $\Phi = \otimes_v \Phi_v \in I_r(s, \chi_{V})$ is factorizable section. Then:
		\begin{enumerate}[(i)]
			\item $E_T(g, s, \Phi) = \prod_{v \leq \infty} W_{T,v}(g_v, s, \Phi_v)$.
			\item Suppose $\Phi_v$ is in the image of the Rallis map and that $V_v$ does not represent $T$ at some place $v$. Then $	W_{T,v}(e, 0, \Phi_v) = 0.$
		\end{enumerate}
		\qed
	\end{proposition}

	For our purposes, we will also require certain auxiliary sections. Fix a prime $p$, let $B^{(p)}$ denote the rational quaternion algebra ramified exactly at $p$ and $\infty$, and let $V^{(p)} \subset B^{(p)}$ denote the subset of traceless elements. We equip $V^{(p)}$ with the quadratic form $Q(x) = \mathrm{nm}(x)$, the reduced norm. We fix an order $\calO^{(p)} \subset \calB^{(p)}$ as follows: if $p \nmid N$, then we take $\calO^{(p)}$ to be an Eichler order of level $N$. If $p | N$, we take $\calO^{(p)}$ to be an Eichler order of level $N/p$. In any case, let
	\begin{equation} \label{eqn:Lp lattice def}
	L^{(p)} := \calO^{(p)} \cap V^{(p)}
	\end{equation}
	denote the set of trace-zero elements.
	
	Finally, we set
	\begin{equation} \label{eq:Neighbor}
	\Phi^{(p)}(s) \in I_2(s,\chi')
	\end{equation}
	to be the global standard section associated to  $(L^{(p)})^{\oplus 2}$; here $\chi' = \chi_{V^{(p)}}$ is the character $\chi'(x) = (-1, x)_{\bbA}$.

	\begin{lemma} \label{lem:Whittaker values shift}
		Suppose $T \in \Sym_2(\bbZ)^{\vee}$, fix a prime $p$ as above,  and let $q \neq p$.
		\begin{enumerate}[(i)]
			\item 	If $q \nmid N$, then
			\begin{equation*}
			W_{NT, q}(e, s, \Phi^{(p) }_q ) \ = \  \left( \frac{\gamma_q(V^{(p)}_q)}{\gamma_q(\calV_q)} \right)^2 W_{T, q}(e,s, \Phi^{\calL}_q).
			\end{equation*}
			\item Suppose $q | N$. Then at $s=0$, we have
		\end{enumerate}
		\begin{equation*}
		W_{NT, q}(e, 0, \Phi^{(p) }_q ) \ = \  \left( \frac{\gamma_q(V^{(p)}_q)}{\gamma_q(\calV_q)} \right)^2 W_{T, q}(e,0, \Phi^{\calL}_q).
		\end{equation*}
		\begin{proof}
			Recall that
			\begin{equation}
			W_{T,q}(e, s, \Phi^{\calL}_q) = \gamma_q(\calV_q)^2 \cdot |\det \calL_q|_q \cdot \alpha_q(X, T, \calL_q)|_{X= p^{-s}}
			\end{equation}
			and similarly for $W_{NT,q}(e, s, \Phi^{(p)}_q)$.
			
			If $q \nmid N$, then we have identifications $\calL_q \simeq (M_2(\bbZ_q)^{tr=0}, N \det(x))$, and $L^{(p)}_q \simeq (M_2(\bbZ_q)^{tr=0},  \det(x))$ as quadratic spaces. By \cite[Proposition 8.6]{YangLocal},  we have
			\begin{equation}
			\alpha_q(X, T, \calL_q) =\alpha_q(X, N^{-1} T, L^{(p)}) =\alpha_q(X, N T, L^{(p)}_q)
			\end{equation}
			from which the first part of the lemma follows.

			When $q | N$, we have that
			\begin{equation}
			L_q^{(p)} = \left\{ \begin{pmatrix} a & b \\ qc & -a \end{pmatrix} \ | \ a,b,c \in \bbZ_q  \right\}, \qquad Q(x) = \det (x).
			\end{equation}
			By \cite[Corollary 8.5]{YangLocal}, we have
			\begin{equation} \label{eq:Lp}
			\alpha(X, NT, L_q^{(p)}) = q^2 \alpha(X, NT, L_0) + (q-q^2) R_1(X, NT, L_0) + 1 - q^2,
			\end{equation}
			and so, comparing with \Cref{lem:rep dens relation}, taking $X =1 $ in both formulas gives
			\begin{equation}
			\alpha(1, NT, L_q^{(p)}) = q \alpha(1, T, \calL_q).
			\end{equation}
			Observing that $|\det \calL_q|_q = q^{-1} $ and $|\det L^{(p)}_q|_q = q^{-2}$, the second part of the  lemma follows immediately.
		\end{proof}
		
	\end{lemma}
	
	\begin{corollary} \label{cor:nondeg Eis shift}
		Suppose $\mathrm{Diff}(T) = \{p\}$. Then
		\begin{equation*}
		E'_T(g_{\tau}, 0, \Phi^{\calL}) \cdot q^{-T} = \frac{\nu_p(NT) \cdot \log p}{2 c_p}   \cdot E_{NT}(g_{\tau}, 0, \Phi^{(p)}) \cdot q^{-NT}
		\end{equation*}
		
		\begin{proof}
			By \Cref{prop:Whittaker dichotomy}, we have
			\begin{equation}
			W_{T,p}(e, 0, \Phi^{\calL}_p) = 0,
			\end{equation}
			and so
			\begin{equation}
			E'_T(g_{\tau}, 0, \Phi^{\calL}) = W'_{T,p}(e, 0, \Phi_{p}^{\calL}) \cdot  W_{T,\infty}(g_{\tau}, 0, \Phi^{\calL}_{\infty}) \cdot \prod_{ \substack { v < \infty \\ v \neq p}} W_{T,v}(e, 0, \Phi_{v}^{\calL}).
			\end{equation}
			At the infinite place, we have that $\Phi_{\infty}^{\calL} = \Phi_{\infty}^{3/2} = \Phi_{\infty}^{(p)}$. The corresponding Whittaker functionals are given explicitly by the formula
			\begin{equation}
			W_{T,\infty}(g_{\tau}, 0, \Phi_{\infty}^{3/2} )  = - 2 \sqrt{2}(2 \pi)^2 \det(v)^{3/4} q^T,
			\end{equation}
			cf.\  \cite[Theorem 5.2.7(i)]{KRYbook}, for any non-degenerate $T \in \Sym_2(\bbQ)$. Thus, in our case we find
			\begin{equation} \label{eq:infty}
			\frac{W_{T, \infty}(g_{\tau}, 0, \Phi_{\infty}^{\calL})}{W_{NT, \infty}(g_{\tau}, 0, \Phi_{\infty}^{(p)})} = \frac{q^T}{q^{NT}}.
			\end{equation}
			Combining this identity with \Cref{lem:Whittaker values shift} and \Cref{prop:local intersection number}, we find
			\begin{align}
			E_T'(g_{\tau}, 0, \Phi^{\calL}) q^{-T} &= \frac{\nu_p(NT) \log p}{2 c_p} \cdot \left( \prod_{v < \infty} \frac{\gamma_v(\calV)^2}{\gamma_v(V^{(p)})^2} \right)  \cdot \prod_{v \leq \infty} W_{NT,v}(g_{\tau}, 0, \Phi^{(p)}) \cdot q^{-NT} \\
			&= \frac{\nu_p(NT) \log p}{2 c_p}  \left( \prod_{v < \infty} \frac{\gamma_v(\calV)^2}{\gamma_v(V^{(p)})^2} \right)   E_{NT}(g_{\tau}, 0, \Phi^{(p)}) \,  q^{-NT}.
			\end{align}
			It remains to show that the product involving Weil indices equals 1. For any quadratic space $V$, the product formula $\prod_{v \leq \infty} \gamma_v(V) = 1$ holds; thus
			\begin{equation}
			\prod_{v < \infty} \frac{\gamma_v(\calV)^2}{\gamma_v(V^{(p)})^2}  = \frac{\gamma_{\infty}(\calV)^2}{\gamma_{\infty}(V^{(p)})^2}.
			\end{equation}
			By \cite[p.\ 330]{KRYbook}, we have $\gamma_{\infty}(\calV)^2 = -1 = \gamma_{\infty}(V^{(p)})^2$, so the ratio above is one, and the corollary is proved.
		\end{proof}
	\end{corollary}

	The next step in our proof is to apply the Siegel-Weil formula for the positive definite space $V^{(p)}$. Let $H^{(p)} = O(V^{(p)})$, viewed as an algebraic group over $\bbQ$. If
	\begin{equation}
	\varphi^{(p)} \in S((V^{(p)} \otimes_{\bbQ} \bbA_f)^{\oplus 2})
	\end{equation}
	is the characteristic function of $(L^{(p)} \otimes_{\bbZ} \widehat \bbZ)^{\oplus 2}$, we  define the theta integral
	\begin{equation}
	I(g, \varphi^{(p)}) = \int_{[H^{(p)}]} \ \Theta(g, h, \varphi^{(p)}) \, dh, \qquad g \in G_{\bbA}
	\end{equation}
	here $\Theta(g, h, \varphi^{(p)})$ is the usual theta function, and $dh$ is the left Haar measure on $[H^{(p)} ] = H^{(p)}(\mathbb Q) \backslash H^{(p)}(\bbA)$ normalized so the total volume is 1.
	
	Then the Siegel-Weil formula implies that for any non-degenerate $T$, we have
	\begin{equation}
	E_{T}(g, 0, \Phi^{(p)}) = 2 \,  I_{T}(g,  \varphi^{(p)}).
	\end{equation}

	The computation of $I_T(g,\varphi^{(p)})$ is given  in \cite[\S 5.3]{KRYbook}; we review the computation here. Note that $O(V^{(p)}) \simeq SO(V^{(p)}) \times \mu_2$, and since $\varphi^{(p)}$ is the characteristic function of a lattice, it is $\mu_2(\bbA_f)$-invariant. The measure $dh$ decomposes as $dh_1 \times dc$ where the volume of $\mu_2(\bbQ) \backslash \mu_2(\bbA)$ with respect to $dc$ is equal to $1/2$; fixing a gauge form $\omega$ on $SO(V^{(p)})$ as in \cite[p.\ 126]{KRYbook}, we obtain a decomposition $dh_1 = \prod_{v \leq \infty} dh_{1,v}$.

	There is a surjective map $B^{(p)} \to SO(V^{(p)})$, where an element of $B^{(p)}$ acts by conjugation on $V^{(p)}$.
	Let $K^{(p)} \subset SO(V^{(p)})(\bbA_f)$ denote the image of $(\widehat{\calO^{(p)}})^{\times}$ under this map, and write
	\begin{equation} \label{eqn:weak approx SO(Vp)}
	SO(V^{(p)})(\bbA_f) = \coprod_j SO(V^{(p)})(\bbQ) \cdot h_j \cdot K^{(p)}.
	\end{equation}
	Let $\Gamma_j = SO(V^{(p)})(\bbQ) \cap  h_j K^{(p)} h^{-1}_j$. Then \cite[Proposition 5.3.6]{KRYbook} implies that for $g = (g_{\infty}, e, \dots)$,
	\begin{equation} \label{eqn:theta coeff}
	I_T(g,\varphi^{(p)}) = \frac12 O_{T,\infty}(g_{\infty},\varphi^{(p)}_{\infty}) \,  \mathrm{vol}\left( K^{(p)}, dh_{1,f} \right) \,
	\left( \sum_j \sum_{\substack{ y \in \Omega(T, V^{(p)}) \\ \text{mod } \Gamma_j }} \varphi^{(p)}(h_j^{-1} y) \right)
	\end{equation}
	where
	\begin{equation}
	O_{T,\infty}(g_{\infty}, \varphi'_{\infty}) = \int_{SO(V^{(p)})(\bbR)} \omega(g_{\infty}) \varphi'_{\infty}(h_{1,\infty}^{-1} \cdot x_0) \  dh_{1,\infty}
	\end{equation}
	and $x_0 \in \Omega(T,V^{(p)})$ is fixed.
	
	The volume appearing in \eqref{eqn:theta coeff} can be computed following \cite[Lemma 5.3.9]{KRYbook}; we sketch the argument here.
	
	\begin{lemma} \label{lemma:local volumes}
		Suppose $\mathrm{Diff}(T) = \{p\}$. Then
		\begin{equation}
		O_{T,\infty}(g_{\tau}, \varphi^{(p)}_{\infty})  \ \mathrm{vol}(K^{(p)}, dh_{1,f}) = \frac{24}{p-1}    \cdot \left( \prod_{\substack{q|N \\ q \neq p}} (1+q)^{-1} \right) \cdot \det(v)^{3/4} \, q^T.
		\end{equation}
		\begin{proof}
			Suppose $v \leq \infty $ and $\varphi_v \in S((V^{(q)}_v)^2)$ is any Schwartz function, and define the local orbital integral
			\begin{equation}
			O_{T,v}(g_v, \varphi_v) :=  \int_{SO(V^{(q)})(\bbQ_v)}  \omega(g_v) \varphi_v(h^{-1} x_0) \   dh_{1,v}
			\end{equation}
			where $x_0 \in \Omega(T, V^{(q)})$ is a fixed tuple in $(V^{(q)})^2$ with $T(x_0) = T$.
			Then there is a non-zero constant $d_v = d_{v}(V^{(q)}, dh_{1,v} )$ such that
			\begin{equation}
			O_{T,v}(g_v, \varphi_v)  \ = \  d_{v} \, W_{T,v}(g_v, 0, \Phi(\varphi_v)).
			\end{equation}
			By \cite[Proposition 5.3.3]{KRYbook}, this constant only depends on the local measure $dh_{1,v}$, and not on $T$, and moreover
			\begin{equation}
			\prod_{v \leq \infty}d_v  = 1.
			\end{equation}
			Now arguing as in \cite[Lemma 5.3.9]{KRYbook}, we have that for a finite prime $q$,
			\begin{equation}
			\mathrm{vol}(K_q^{(p)}, dh_{1,q}) = d_q \cdot \gamma_q(V^{(p)})^2 \cdot |2|_{q}^{3/2}
			\cdot 	\begin{cases}
			(1 - q^{-2}), & 				q \nmid Np \\
			(1+q )^{-1} (1-q^{-2}),&		q |N, \,  q \neq p	 \\
			p^{-2} (p+1), &						 q = p.
			\end{cases}
			\end{equation}
			Thus, we have
			\begin{align*}
			O_{T,\infty}&(g_{\tau,\infty}, \varphi^{(p)}_{\infty}) \cdot 	\mathrm{vol}(K^{(p)}, dh_{1,f}) \\
			&=   W_{T, \infty}(g_{\tau, \infty}, 0, \Phi^{3/2}_{\infty})
			\left( \prod_{v < \infty} \gamma_v(V^{(p)})^2 |2|_v^{3/2} \right)
			\left( \prod_{q < \infty}(1-q^{-2}) \right) \\
			& \qquad \times
			\left( \frac{p^{-2}(p+1)}{1-p^{-2}} \right)
			\left( \prod_{\substack{q|N \\ q \neq p}} (q+1)^{-1} \right)
			 \\
			&=  W_{T, \infty}(g_{\tau, \infty}, 0, \Phi^{3/2}_{\infty})
			\cdot \left(\gamma_{\infty}(V^{(p)})^{-2} \, 2^{-3/2}  \right)  \zeta(2)^{-1} \frac{1}{p-1}\left( \prod_{\substack{q|N \\ q \neq p}} (q+1)^{-1} \right).
			\end{align*}
			Finally, we recall that $\gamma_{\infty}(V^{(q)})^2 = -1$, see \cite[eqn.\ (5.3.71)]{KRYbook} and $\zeta(2) = \pi^2 / 6$, and	
			\begin{equation}
			W_{T,\infty}(g_{\tau}, 0, \Phi_{\infty}^{3/2} )  = - 2 \sqrt{2}(2 \pi)^2 \det(v)^{3/4} q^T,
			\end{equation}
			cf.\  \cite[Theorem 5.2.7(i)]{KRYbook}; this proves the desired formula.
		\end{proof}
	\end{lemma}
	
	Recall that we had written the Fourier expansion of the Eisenstein series $E(\tau,s,\Phi^{\calL})$ as
	\begin{equation}
	E(\tau, s, \Phi^{\calL}) = \sum_T C_T(v, s, \Phi^{\calL}) q^T.
	\end{equation}
	
	Combining \Cref{cor:nondeg Eis shift} and \Cref{lemma:local volumes} (replacing $T$ by $NT$ in the latter lemma) we obtain the following result.
	\begin{corollary} \label{cor:eisenstein pos def}
		Suppose $T>0$ and $\mathrm{Diff}(T) = \{ p\}$.Then $C'_T(0, \Phi^{\calL}) = C'_T(0,v, \Phi^{\calL})$ is independent of $v$, and is given by the formula
		\begin{equation}
		C'_T(0, \Phi^{\calL})
		= \ \nu_p(NT)  \log p \cdot \left( 12\, \prod_{q | N}(1+q)^{-1} \right)   \cdot \ \left( \sum_j \sum_{\substack{ x \in \Omega(NT, V^{(p)}) \\ \text{mod } \Gamma_j }} \varphi^{(p)}(h_j^{-1} x) \right)
		\end{equation}
	\end{corollary}
	
	Finally, we prove the main identity in the positive definite case:
	\begin{theorem} \label{thm:main thm pos def} Suppose $T>0$. Then
		\begin{equation}
		\widehat \deg \,  \calZ(T)  = \frac{ \prod_{q|N}(q+1)}{24} \cdot C'_T(0,\Phi^{\calL})
		\end{equation}
		
		\begin{proof}
			First, suppose $\# \mathrm{Diff}(T) > 1$. By \Cref{lemma:support of Z(T)}, the left hand side vanishes, and by \Cref{prop:Whittaker dichotomy}, the right hand side vanishes, establishing the result in this case.

			We may therefore suppose $\mathrm{Diff}(T) = \{ p\}$ for some prime $p$.
			Recall that there is an identification\footnote{Briefly, the right hand side is interpreted as the set of   invertible right $\calO^{(p)}$-modules. This latter set is in bijection with $\calX(\overline \bbF_p)^{ss}$ as follows. Fix a base point $(\pi_0 \colon E_0 \to E'_0 )\in \calX(\overline \bbF_p)^{ss}$. Then, given a point $\pi \colon E \to E' \in \calX(\overline \bbF_p)^{ss}$, the corresponding $\calO^{(p)}$ module is $\Hom(\pi_0, \pi)$. See e.g.\ \cite[\S 3]{Ribet100} for details.}
			\begin{equation}
			\calX (\bar{\bbF}_p)^{ss} \ \simeq  \ \left[ B^{(p), \times} \Big\backslash \left( B^{(p),\times}(\bbA_f) \Big/   \widehat{ \calO^{(p)}}  {}^{\times} \right) \right].
			\end{equation}
			Now suppose $\varphi \colon E \to E' \in \calX(\overline \bbF_p)^{ss}$ is a geometric point corresponding to the coset $[b]  = b\, \widehat{\calO^{(p)}}{}^{\times}$ as above. Then the lattice $L(\varphi)$, defined in \eqref{eqn:moduli lattice def},  is identified with the lattice
			\begin{equation}
			b \cdot L^{(p)}  :=  \left( b  \, \widehat{ L^{(p)}} b^{-1} \right)	 \cap V^{(p)}	
			\end{equation}
			where $L^{(p)} = V^{(p)} \cap \calO^{(p)}$ as above.
			
			Tracing through definitions, we have an identification
			\begin{equation}	
			\calZ(T)(\overline \bbF_p) \simeq \left[ B^{(p), \times} \Big \backslash \calC(T) \right]			
			\end{equation}
			where 			
			\begin{equation}
			\calC(T) = \left\{ ( y, [b]) \in  (V^{(p)})^{ 2} \times   B^{(p),\times}(\bbA_f) \Big/   \widehat{ \calO^{(p)}}  {}^{\times}  \Big | \ y \in (b\cdot L^{(p)})^{ 2},\  T(y) = NT  \right\}.
			\end{equation}
			As described in  \Cref{sec:local}, the (arithmetic) degree of the local ring of $\calZ(T)$ at each geometric point is the same, and is given by $\nu_p(T) \log p$. Hence
			\begin{align}
			\widehat \deg \, \calZ(T)  &= \sum_{z \in \calZ(T)(\bar{\bbF}_p )}  \frac{ \log |  \calO_{\calZ(T),z}| }{| \Aut(z)|} \\	
			&=	 \nu_p(T) \log p \cdot \# \left[ B^{(p), \times} \big\backslash \calC(T)\right],
			\end{align}
			where on the right hand side, we have the ``stacky" cardinality, i.e.\ the number of orbits, with each orbit weighted by the reciprocal of the order of the corresponding stabilizer group.
			
			To determine this cardinality, let $H^{(p)} = SO(V^{(p)})$, and recall that there is an exact sequence
			\begin{equation}
			\begin{CD}
			1 @>>> Z @>>> B^{(p),\times} @>c>> H^{(p)} @>>> 1
			\end{CD}
			\end{equation}
			where $b \in B^{(p),\times}$ acts on $V^{(p)}$ by conjugation, and $Z \simeq \mathbb G_m$ is the centre. Since $Z(\bbA_f) = Z(\bbQ) \cdot \left( Z(\bbA_f) \cap ( \widehat{ \calO^{(p)}})^{\times} \right) $, we have a bijection
			\begin{equation} \label{eqn:ss points double coset}
			B^{(p), \times} \Big\backslash  B^{(p),\times}(\bbA_f) \Big/   \widehat{ \calO^{(p)}}  {}^{\times} \longleftrightarrow  H^{(p)} (\bbQ) \Big\backslash H^{(p)} (\bbA_f) \Big/ K^{(p)}
			\end{equation}
			where $K^{(p)}$ was, by definition, the image of $(\widehat{ \calO^{(p)})}{}^{\times}$ in $H^{(p)} (\bbA_f)$. As in \eqref{eqn:weak approx SO(Vp)}, we choose representatives $\{ h_k\}$ for this double coset space, and set $  \Gamma_j = H^{(p)}(\bbQ) \cap  h_j K^{(p)} h^{-1}_j$.
			Moreover, since the components of $y$ span a two dimensional subspace of $V^{(p)}$, it follows that the stablilizer of any point $(y, [b])$ is equal to $Z(\mathbb Q) \cap \widehat{\calO^{(p)}}{}^{\times} = \{ \pm 1 \}$. Thus
			\begin{align}
			\# \left[ B^{(p), \times} \big\backslash \calC(T)\right] \ =  \ \frac12 \sum_j \, \sum_{\substack{y \in h_j \cdot L^{(p)} \\ T(y) = NT  \\ \text{mod} \Gamma_j }} 1 \ =  \ \frac12  \sum_j \, \sum_{\substack{y \in \Omega(NT, V^{(p)}) \\ \text{mod } \Gamma_j }} \varphi^{(p)}(h_j^{-1} y)
			\end{align}
			where $\varphi^{(p)} \in S(V^{(p)}(\bbA_f)^2)$ is the characteristic function of $(\widehat{ L^{(p)}})^{2}$.	The theorem follows from a comparison with \Cref{cor:eisenstein pos def}.
		\end{proof}
	\end{theorem}

	\subsection{ $T$ of signature $(1,1)$ or $(0,2)$}
	Recall that when $T$ is non-degenerate but not positive definite, the class $\widehat{\calZ}(T,v)$ is purely archimedean:
	
	\begin{equation}
	\widehat \calZ(T,v) \ := \ \left( 0,  \Xi(T,v) \right) \ \in  \ \ChowHat{2}_{\bbC}(\calX_0(N))
	\end{equation}
	where $\Xi(T,v)$ is defined in \eqref{eqn:Green current indefinite case}. In this case, the arithmetic degree is given by
	\begin{equation} \label{eqn:arch degree integral}
	\widehat \deg \, \widehat{ \mathcal Z}(T,v)  \ = \ \frac12  \int_{\calX(\bbC)} \Xi(T,v) .
	\end{equation}
	The computation of this integral was carried out in \cite{GarciaSankaran} in the case of compact Shimura varieties. A crucial step in the argument is the application of the Siegel-Weil formula to relate the integral of a certain theta function (attached to the Schwartz form denoted by $\nu(x)$ in \emph{loc.\ cit.}) to a special value of the corresponding Eisenstein series.
	
	In our case, the fact that $\calV$ contains isotropic vectors implies that the relevant theta integrals do not necessarily converge for  arbitrary Schwartz functions; however, as we prove below, the particular Schwartz form relevant to our setting \emph{does} lead to a convergent theta integral, and work of Kudla and Rallis \cite{KudlaRallisRegSW} implies that the Siegel-Weil formula holds for this Schwartz form. With this fact established, the arguments of \cite{GarciaSankaran} carry through verbatim and yield a computation of the integral \eqref{eqn:arch degree integral}.
	
	To state things more precisely, recall  the explicit expression 			
	\begin{equation}
	\nu(\lambda) = \psi(\lambda) \, d \mu
	\end{equation}
	where
	\begin{equation}
	\psi(\lambda,z) =  \left(- \frac{1}{\pi} +  2 \sum_{i =1}^2( R(\lambda_i,z) + 2 Q(\lambda_i)) \right) e^{-2\pi \sum R(\lambda_i, z) + Q(\lambda_i)}, \qquad \lambda= (\lambda_1, \lambda_2) \in \calV_{\bbR}^2,
	\end{equation}
	and
	\begin{equation}
	d \mu = \frac{dx \wedge dy}{y^2}
	\end{equation}
	with $z = x + iy \in \bbH^{\pm}$.
	
	For future use, we fix the basepoint $i \in \mathbb H$ and abbreviate
	\begin{equation} \label{eqn:Schwartz function phi}
	\phi(\lambda) = \psi(\lambda, i)
	\end{equation}
	so that $\phi\in S(\calV_{\bbR}^2)$ is a Schwartz function.

	This Schwartz function gives rise to a theta function that, at least for the moment, is best described adelically. Let $H= O(\calV)$, and let $\omega$ denote the action of $G_{2,\mathbb A} \times H(\mathbb A)$  on $S(V^2(\mathbb A))$ via the Weil representation. 	Recall that the action of $H(\mathbb A)$ is the linear action; more precisely, for $h \in H(\mathbb A)$ ,  $\lambda = (\lambda_1, \lambda_2) \in \calV(\bbA)^2$ and $\varphi \in S(\calV(\bbA)^2)$, we have
	\begin{equation}
	\omega(1, h)  \varphi(\lambda) = \varphi( h^{-1} \cdot \lambda) = \varphi \left( h^{-1} \cdot \lambda_1, \,  h^{-1} \cdot \lambda_2  \right) .
	\end{equation}
	
	Returning to our Schwartz function $\phi$, we consider the adelic Schwartz function $\phi_{\bbA}=\phi\otimes 1_{\widehat \omega}$, where $1_{\widehat \omega} \in S(\calV(\bbA_f)^2)$ is the characteristic function of $\widehat \calL^2 = (\calL\otimes_{\bbZ} \widehat{\bbZ})^{ 2}$. Then the corresponding theta series is
	\begin{equation}
	\Theta(g',h,  \phi_{\bbA}) :=  \sum_{\lambda \in \calV^2} \omega(g', h)  \phi_{\bbA} (\lambda).
	\end{equation}

	\begin{lemma}
		As a function of $h$, the theta function $\Theta(g', h, \phi)$ is left $H(\bbQ) $ invariant, is absolutely convergent, and descends to a rapidly decreasing function on $H(\bbQ) \backslash H(\bbA)$.
		\begin{proof}
			The left invariance is immediate from definitions. To prove the convergence and rapid decrease, we make use of the ``mixed model" of the Weil representation. In general, suppose $V$ is a quadratic space of Witt rank $r$, decomposed as
			\begin{equation}
			V = V_{an} \stackrel{\perp}{\oplus} V_{r,r}
			\end{equation}
			where $V_{an}$ is anisotropic and $V_{r,r}$ is isomorphic to $r$ copies of the hyperbolic plane. Fix a standard basis
			\begin{equation}
			v_1', \dots, v_r', v_1'', \dots, v_r''
			\end{equation}
			where $\langle v_i', v_i'' \rangle = 1$ for $i = 1, \dots, r$, and all other inner products are zero. Setting $V' = \mathrm{span} \{  v_1', \dots, v_r'\}$,  the given basis yields an identification $(V')^k \simeq M_{k, r}(\bbQ)$ for any positive integer $k$ by the map
			\begin{equation}
			A = (a_{ij}) \in M_{k,r}(\bbQ)  \ \longleftrightarrow \  \left(   \sum_{i=1}^r a_{1i} v_i', \dots \sum_{i = 1}^r a_{ki} v_i'  \right).
			\end{equation}
			We may also identify $(V'')^k \simeq M_{k,r}(\bbQ)$ in a similar way.
			
			For $G' = \Mp_4$, the mixed model  of the Weil representation $\widehat \omega$ is an action  $G'_{\bbA} \times O(V)(\bbA)$ on $S(V(\bbA)^2) \simeq S(V_{an}(\bbA)^{2})  \otimes S( M_{2,r}(\bbA)) \otimes S(M_{2,r}(\bbA)) $. The two models are related by an intertwining map
			\begin{equation}	
			S(V(\bbA)^2) \stackrel{\sim}{\longrightarrow}  S(V_{an}(\bbA)^{2})  \otimes S( M_{2,r}(\bbA)) \otimes S(M_{2,r}(\bbA)), \qquad \varphi \mapsto \widehat \varphi,
			\end{equation}
			where, by definition,
			\begin{equation}
			\widehat \varphi(\lambda, A, B) \ := \ \int_{M_{2,r}(\bbA) }  \varphi(\lambda + X + B) \psi(\mathrm{tr}A^t X) dX,
			\end{equation}
			where $\psi \colon \bbA \to \bbC$ is the standard additive character that is trivial on $\bbQ$ and $\widehat{\bbZ}$, and we identify $V'(\bbA)^2 \simeq M_{2,r}(\bbA) \simeq V''(\bbA)^2$ as above.
			
			With this setup in place, we have the following criterion due to Kudla-Rallis:  let $\varphi \in V(\bbA)^2$ and suppose that
			\begin{equation}  \label{eqn:theta convergence criterion}
			\widehat \varphi( x, A, B) = 0 \qquad \text{whenever } \mathrm{rank}([A \, B]) < r.
			\end{equation}
			Then the proof of \cite[Proposition 5.3.1]{KudlaRallisRegSW} implies that $\Theta(g',h, \varphi)$ is absolutely convergent and rapidly decreasing as a function of $h \in H(\bbQ) \backslash H(\bbA)$.

			Returning to the case at hand, note that the Witt rank of $\calV$ is $r=1$, and we have an orthogonal basis
			\begin{equation}
			e =   \psm{1 & \\ & -1}, \  f_1 = \psm{0 &  0 \\ 1 & 0}, \ f_2 = \psm{ 0 & 1/N \\ 0 & 0 }
			\end{equation}
			with $\calV_{an} = \bbQ e$ and $\calV' = \bbQ f_1$, and $\calV'' = \bbQ f_2$. Applying the criterion above, it will suffice to show that
			\begin{equation}
			\widehat \phi( \lambda_{an}, 0, 0)  = \int_{(x_1, x_2) \in \bbR^2}  \phi(   a_1e + x_1 f_1, a_2 e + x_2 f_1)  \, dx_1 \, dx_2
			\end{equation}
			vanishes for all $\lambda_{an} = (a_1 e, a_2e) \in \calV_{an}(\bbR)^2$, for the Schwartz function $\phi$ defined in \eqref{eqn:Schwartz function phi}.
			
			We compute
			\begin{equation}
			R \left( a e + x f_1 , i \right) = R \left( \begin{pmatrix} a & \\ x & -a \end{pmatrix}, \ i \right)  = \frac{N(x^2 + 4 a^2)}{2}
			\end{equation}
			and $Q(a e + x f_1) = -N a^2$. Unwinding definitions, we find that
			\begin{align}
			\widehat \phi( \lambda_{an}, 0, 0) =  e^{ - 2 \pi N (a_1^2 + a_2^2)}  \int_{\bbR^2}  \left( - \frac{1}{\pi} + N( x_1^2 + x_2^2) \right) e^{- \pi N (x_1^2 + x_2^2) }\, dx_1 \, dx_2,
			\end{align}
			and a straightforward calculus exercise shows that the integral vanishes, as required.
		\end{proof}
	\end{lemma}
	
	\begin{corollary}[{\cite{KudlaRallisRegSW}}] \label{cor:Siegel Weil nu} Let $\Phi_{\phi} \in I_2(s, \chi_{\calV})$ denote the standard section attached to $\phi_{\bbA}$, and let $E(g, s, \Phi_{\phi})$ denote the corresponding Siegel Eisenstein series. Then
		\begin{equation}
		2  	\int_{[H]} \Theta(g,h,\phi_{\bbA}) dh =  E(g,  0, \Phi_{\phi})
		\end{equation}
		where $dh$ is the Haar measure on $[H] = H(\bbQ) \backslash H(\bbA)$ giving total volume 1.  \qed
	\end{corollary}
	\begin{remark} Strictly speaking, \cite{KudlaRallisRegSW} asserts that the two sides are proportional; the constant of proportionality is shown to be 2 in \cite[Theorem 7.3(ii)]{GanQiuTakeda}, see also the references therein. \end{remark}
	
	\begin{theorem} Suppose $T$ is of signature $(1,1)$ or $(0,2)$. Then
		\[
		\widehat{\deg}  \, \widehat{\mathcal Z}(T,v) =  \frac{\prod_{q|N}(q+1)}{24}   C_{T}'(v,0, \Phi^{\calL}).
		\]
		
		\begin{proof}
			By definition,  the  degree is
			\begin{equation}
			\widehat \deg \, \widehat{ \mathcal Z}(T,v)  \ = \ \frac12  \int_{\calX(\bbC)} \Xi(T,v) .
			\end{equation}
			
			With the Siegel-Weil formula for the Schwartz function $\phi_{\bbA}$, \Cref{cor:Siegel Weil nu}, in hand, the proof of \cite[Theorem 5.10]{GarciaSankaran} goes through verbatim; in our setting, this formula reads:
			\begin{equation}
			\int_{\calX(\bbC)} \Xi(T,v) =   \mathrm{vol}(\calX(\bbC), d \Omega)   \ C'_T(v,0, \Phi^{\calL})
			\end{equation}
			where $d \Omega =\frac{1}{2\pi} du dv / v^2$. The volume is given explicitly by the well-known formula
			\begin{equation}
			\mathrm{vol}(\calX(\bbC), d \Omega) = \frac{1}{ 2} \int_{\Gamma_0(N) \backslash \bbH}  \frac{du \, dv}{ 2 \pi v^2} =  \frac12 \cdot \frac{\prod_{q|N}(q+1)}{6},
			\end{equation}
			where the factor $\frac12$ emerges from the action of the group $\{\pm 1 \}$, cf. \Cref{rmk:stacky 1/2}. The theorem follows immediately.
		\end{proof}
	\end{theorem}
	
	\subsection{$T$ of rank one}

	Here, we follow \cite[\S 5.8]{KRYbook}; see also \cite[\S 5.2]{GarciaSankaran} for a more general discussion. We will make use of Eisenstein series in genus one and two, see \Cref{sec:main theorem setup} for the notation.
	
	We begin with the observation that there is an embedding
	\begin{equation}
	\eta \colon G_{1,\bbA} \to G_{2,\bbA}, \qquad  \left[ \begin{pmatrix} a & b \\ c & d \end{pmatrix}, z \right] \mapsto \left[ \begin{pmatrix} 1 & & 0 & \\ & a & & b  \\ 0 & &  1 & \\ & c & & d \end{pmatrix}, z \right]
	\end{equation}
	which induces a map
	\begin{equation} \label{eqn:eta defn}
	\eta^*  \colon  I_{2}(s, \chi_{\calV}) \to I_{1}(s+\tfrac12, \chi_{\calV})
	\end{equation}
	via pullback. By the same formula, we have an embedding $\eta_v \colon G_{1, v} \to G_{2,v}$ at each place $v$, inducing a pullback $\eta^*_v \colon I_{2, v}(s, \chi_{\calV}) \to I_{1, v}(s+\tfrac12, \chi_{\calV})$.

	\begin{lemma}[{\cite[Lemma 5.4]{GarciaSankaran}}]  \label{lem:GS Eis FC rk 1}
		Suppose
		\begin{equation}
		T = \begin{pmatrix} 0 & \\ & t \end{pmatrix}.
		\end{equation}
		with $t \neq 0$.
		Then for any $\Phi \in I_2(\chi, s)$ and $g \in G_{2,\bbA}$, we have
		\begin{equation}  \label{eqn:Eisenstein FC rank 1}
		E_T(g, s, \Phi) =  W_t(e, s+\tfrac12,  \left(\eta^* \circ r(g)\right) \Phi) + W_T(g, s, \Phi)
		\end{equation}
		\qed
	\end{lemma}

	Let's compute these Whittaker functionals more explicitly, in the case that
	\begin{equation}
	\Phi = \Phi_2^{\calL} \in I_2(\chi, s)
	\end{equation}
	is the incoherent section from \Cref{def:calL global section} above. We begin with a simple lemma:
	
	\begin{lemma} \label{lem:pullback on Rallis sections}
		Suppose $V_v$ is a local quadratic space of dimension $m$, and $\varphi_1, \varphi_1' \in S(V_v)$. Let $\varphi_2 = \varphi_1 \otimes \varphi_1' \in S(V_v^{\oplus 2})$, and let $\Phi_2(s) \in I_{2,v}(s, \chi_V)$ and $\Phi_1'(s)  \in I_{1,v}(s, \chi_V)$ denote the  standard sections corresponding to $\varphi_2$ and $\varphi'_1$ respectively. Then
		\begin{equation}
		\eta^* \Phi_2(s) = \varphi_1(0) \cdot \Phi'_1(s+\tfrac12)
		\end{equation}
		
		\begin{proof}
			Let $\omega_i$ denote the Weil representation on $S(V^{\oplus i})$ for $i = 1, 2$. A direct calculation, using the explicit formulas in e.g.\ \cite[Lemma 8.6.5]{KRYbook}, yields the identity
			\begin{equation}
			\left[	\omega_2( \eta(g)) \varphi_2\right](\mathbf v)  = \varphi_1(v) \cdot \left[ \omega_1(g) \varphi_1' \right](v'), \qquad \mathbf v = (v_1, v_1') \in V^{\oplus 2}
			\end{equation}
			for every $g \in G_{1,v}$.
			
			On the other hand, if $g_2 = n(B) m(A) k \in G_{2,v}$, then by definition,
			\begin{equation}
			\Phi_2(g_2,s) = |\det (A)|^{s - \frac{m-3}2} \cdot \left[ \omega_2(g_2) \varphi_2 \right](0)
			\end{equation}
			and for $g_1 = n(b) m(a) k$, we have
			\begin{equation}
			\Phi'_1(g_1,s) = |\det (a)|^{s - \frac{m-2}2} \cdot \left[ \omega_1(g_1) \varphi_1' \right](0)
			\end{equation}
			The lemma follows easily.
		\end{proof}
	\end{lemma}
	
	Turning to the archimedean place, suppose
	\begin{equation}
	v = \psm{ v_1 & \\ & v_2 } \in \Sym_2(\bbR)_{>0},
	\end{equation}
	and let $g_{v, \infty} \in G_{2,\infty}$ denote the corresponding group element as in \eqref{eqn:g_tau def}. Let $\Phi_{2, \infty}^{3/2} \in I_{2,\infty}(s, \chi_{\calV})$ and $\Phi_{1,\infty}^{3/2} \in I_{1, \infty}(s, \chi_{\calV})$ denote the weight $3/2$ sections in genus 2 and 1, respectively. Then a straightforward computation, cf.\ \cite[p.\ 956]{GarciaSankaran}, yields
	\begin{equation} \label{eqn:arch section pullback}
	\left( \eta^* \circ r(g_{v,\infty}) \right) \Phi^{3/2}_{2, \infty}(s) =  v_1^{s/2 + 3/4} \cdot \left( r(g_{v_2,\infty}) \Phi^{3/2}_{1, \infty}(s+1/2) \right).
	\end{equation}

	Now consider the global standard sections $\Phi_1^{\calL} \in I_1(s, \chi_{\calV})$ and  $\Phi_2^{\calL} \in I_2(s, \chi_{\calV})$ as in \Cref{def:calL global section}.

	Then the above discussion implies that for $ v=diag(v_1, v_2)$ as above, and setting
	\begin{equation}
	g_v = (g_{v,\infty}, e,\dots) \in G_{2,\bbA}, \qquad g_{v_2} = (g_{v_2, \infty}, e, \dots)  \in G_{1,\bbA}
	\end{equation}
	we have
	\begin{equation}
	\eta^* \Phi_2^{\calL}(g_v,s) = v_1^{s/2 + 3/4} \, \Phi_1^{\calL}(g_{v_2}', s+\tfrac12)
	\end{equation}
	and hence
	\begin{equation} \label{eqn:whittaker pullback global}
	W_{t}\left( e, s+\tfrac12, \left[ \eta^* \circ r(g_{v}) \right] \Phi_2^{\calL} \right) = v_1^{s/2+3/4} \,   W_t\left( g_{v_2}, s+\tfrac12 , \Phi_1^{\calL} \right).
	\end{equation}
	
	We now turn  to the second term in \eqref{eqn:Eisenstein FC rank 1}; to this end, recall the identities
	\begin{equation}
	W_{T,p}(e, s, \Phi^{\calL}_2) = \gamma_p(\calV)^2 \cdot [\calL_p^{\vee}:\calL_p]^{-1} \cdot |2|_p^{\frac12} \cdot \alpha_p(X,T, \calL)|_{X = p^{-s}}
	\end{equation}
	and
	\begin{equation}	 \label{eqn:whittaker rep dens rank 1}
	W_{t, p}(e, s+ \tfrac12, \Phi^{\calL}_1) = \chi_{\calV,p}(-1)\cdot \gamma_p(\calV) \cdot [\calL_p^{\vee}:\calL_p]^{-\tfrac12} \cdot \alpha_p(X,t, \calL)|_{X = p^{-s}},
	\end{equation}
	cf.\ \cite[Lemma 5.7.1]{KRYbook}. Here
	\begin{equation}
	\calL_p^{\vee} =
	\left\{ x \in \calV_p \ | \ \langle x, y \rangle \in \bbZ_p \text{ for all } y \in \calL_p \right\}
	\end{equation}
	is the dual lattice with respect to the bilinear form $\langle x, y \rangle$ with $\langle x, x \rangle = 2 Q(x) = 2 N \det (x)$, so, recalling that $N$ is odd and squarefree, we have
	\begin{equation}
	[\calL^{\vee}_p: \calL_p] =
	\begin{cases}
	p, & \text{ if } p |2N \\
	1, & \text{ otherwise.}
	\end{cases}
	\end{equation}
	
	For our fixed integer $t \neq 0$, write
	\begin{equation}
	4 Nt = c^2 d
	\end{equation}
	where $-d$ is a fundamental discriminant. Let $\chi_{-d}$ denote the quadratic Dirichlet character attached to the field $k_d := \bbQ(\sqrt{-d})$.
	
	\begin{lemma} \label{lem:Whittaker comparison local} Suppose  $T = \psm{0 & \\ & t}$.
		\begin{enumerate}[(i)]
			\item	If $p \nmid N$ , then
			\begin{equation}
			\frac{W_{T,p}(e, s, \Phi^{\calL}_{2,p})}{W_{t,p}(e, -s + \tfrac12, \Phi_{1,p}^{\calL})} =  \gamma_p(\calV) \, \chi_{\calV, p}(-1) \, |2|_p   \frac{ \zeta_p(-2s+2)}{\zeta_p(2s + 2)} \cdot \frac{ L_p(s, \chi_{-d})}{L_p(1-s, \chi_{-d})} \cdot |c|_p^{2s - 1} \cdot
			\end{equation}
			where $\zeta_p(s) = (1- p^{-s})^{-1}$ and $L_p(s, \chi_{-d}) = (1 - \chi_{-d}(p) p^{-s})^{-1}$ as usual.
			
			\item If $p|N$, then
			\begin{equation}
			\frac{W_{T,p}(e, s, \Phi^{\calL}_{2,p})}{W_{t,p}(e, -s + \tfrac12, \Phi_{1,p}^{\calL})} = - \gamma_p(\calV) \chi_{\calV,p}(-1) \frac{ \zeta_p(s-1)}{\zeta_p(s+1)} \, \frac{ L_p(s,\chi_{-d})}{L_p(1-s, \chi_{-d})} |c|_p^{2s-1} |N|_p^{s - \tfrac12}
			\end{equation}
		\end{enumerate}
	\end{lemma}
	\begin{proof}

		Suppose $p \nmid N$, and let $(0, a)$ denote the Gross-Keating invariants of the matrix $\mathrm{diag}(-N, t) \in \Sym_2(\bbQ_p)$, cf.\ \cite[Appendix B]{YangLocal2}; note that if $p \neq 2$, then $a = \ord_p(t)$.  Using \cite[Proposition B4]{YangLocal2}, one can check that
		\begin{equation}
		\ord_p(c) =
		\begin{cases}
		\frac{a-1}2, & a \text{ odd}, \\
		\frac{a}{2}, & a \text{ even.}
		\end{cases}
		\end{equation}
		In light of \eqref{eqn:Whittaker rep dens relation general}, it suffices to prove the identity
		\begin{equation}
		\left. \frac{\alpha_p(X, T, \calL)}{\alpha_p(X^{-1}, t, \calL)}\right|_{X = p^{-s}} \stackrel{?}{=}   \frac{ \zeta_p(-2s+2)}{\zeta_p(2s + 2)}\cdot \frac{ L_p(s, \chi_{-d})}{L_p(1-s, \chi_{-d})} \cdot |c|_p^{2s - 1},
		\end{equation}
		which amounts to an explicit computation using formulas of Yang. Indeed, by \cite[Proposition C2]{YangLocal2}, we have
		\begin{equation} \label{eqn:alpha rank 1}
		\frac{\alpha_p(X, t, \calL)}{(1-p^{-2}X^2)} =
		\begin{cases}
		\sum_{k=0}^{\frac{a-1}2}  p^{-k} X^{2k}, & a \text{ odd,} \\[10pt]
		\sum_{k=0}^{\frac{a}2}  p^{-k} X^{2k} + (\chi_{-d}(p) - p^{-1}X)^{-1} p^{-\frac{a}2-1}X^{a+1}, & a \text{ even.}
		\end{cases}
		\end{equation}
		On the other hand, taking the limit  $a_3 \to \infty$ in \cite[Theorem 5.7]{YangLocal2}, we have
		\begin{equation}
		\frac{\alpha_p(X,T,\calL)}{1- p^{-2}X^{-2}} =
		\begin{cases}
		\sum_{k=0}^{\frac{a-1}2}  p^{k} X^{2k}, & a \text{ odd,} \\[10pt]
		\sum_{k=0}^{\frac{a}2-1}  p^{k} X^{2k} + p^{\frac{a}2} X^a(1-\chi_{-d}(p)X)^{-1}, & a \text{ even.}
		\end{cases}
		\end{equation}
		The desired relation follows easily from an explicit computation.
		
		For $(ii)$, suppose that $p| N$, so that in particular $p \neq 2$.
		
		Write $a = \ord_p(t)$. Then a direct computation using \cite[Proposition 3.1]{YangLocal} gives
		\begin{equation} \label{eqn:alpha rank 1 level a odd}
		\alpha_p(X, t, \calL) = 1 + (1-p^{-1} ) X \left( \frac{ p^{-\frac{a+1}{2}} X^{a+1} - 1 }{p^{-1}X^2 - 1}   \right) + \chi_{-d}(p) p^{- \frac{a+1}{2}} X^{a+1}
		\end{equation}
		when $a$ is odd, and
		\begin{equation} \label{eqn:alpha rank 1 level a even}
		\alpha_p(X, t, \calL) = 1 + (1-p^{-1} ) X \left( \frac{ p^{-\frac{a}{2}} X^{a} - 1 }{p^{-1}X^2 - 1}   \right)  -  p^{- \frac{a }{2}-1} X^{a+1}
		\end{equation}	
		when $a$ is even.
		
		Consider the case that $a $ is odd, which in turn implies that $\ord_p(c) = \frac{a+1}2$.
		Then taking $b \to \infty$ in the formulas in \Cref{prop:rep dens split} and applying \Cref{lem:rep dens relation} gives
		\begin{multline}
		\frac{X^2}{X-p} \alpha(X, T, \calL) =
		\left( \frac{p^{\frac{a+1}2} X^{a+1} - 1}{p X^2 -1} \right) \left( (p - p^{-1}) X^2 - p^{-1}X -1 \right) \\
	  + \frac{p^{\frac{a+1}{2}} X^{a+1}}{1 - \chi_{-d}(p) X}  \left( - p^{-1} X^2 + (\chi_{-d}(p) - p^{-1})X - 1 \right) + \frac{ X+p}{p};
		\end{multline}
		note that in this case, $\chi_{-d}(p) = (-Nt, p)_p = v_0^+$ in the notation of \Cref{lem:rep dens relation}.	
		
		From this, a straightforward computation shows that
		\begin{equation}
		\frac{\alpha(X,T, \calL)}{\alpha_p(X^{-1}, t, \calL)} = - \left( \frac{1-p^{-1}X}{1-pX} \right) \cdot \left( \frac{1 - \chi_{-d}(p)p^{-1}X^{-1}}{1- \chi_{-d}(p)X} \right)p^{\frac{a+3}2} X^{a+2},
		\end{equation}
		and the lemma follows. A similar computation establishes the case of even $a$.
	\end{proof}
	
	Finally, we have the archimidean counterpart to the previous lemma:
	
	\begin{lemma} \label{lem:Whittaker comparison arch}
		Suppose $T = \psm{ 0 & \\ & t}$ and $v = \psm{ v_1 & \\ & v_2}$. Let
		\begin{equation}
		L_{\infty}(s, \chi_{-d}) := |d|^{\frac{s}2} \cdot
		\begin{cases}
		\pi^{-\frac{s+1}{2}} \Gamma(\frac{s+1}2), & \text{if } d>0, \\
		\pi^{-\frac{s}{2}} \Gamma(\frac{s}2), & \text{if }  d< 0.
		\end{cases}
		\end{equation}
		Then
		\begin{equation}
		\begin{split}
		\frac{W_{T, \infty} \left(g_v, s, \Phi^{\frac32}_{2,\infty}\right)}{W_{t,\infty}\left(g_{v_2}, -s + \frac12, \Phi^{\frac32}_{1,\infty}\right)} &= v_1^{-s/2 + 3/4k} (-2i)  \,  \left( \frac{1-s}{1+s} \right) \cdot
		\frac{ L_{\infty}(s, \chi_{-d})}{L_{\infty}(1-s, \chi_{-d})} \cdot \frac{ \zeta_{\infty}(-2s + 2)}{\zeta_{\infty}(2s + 2) }  \\
		& \qquad \times c^{2s-1} N^{-s + \frac12}.
		\end{split}
		\end{equation}
		Recall here we are writing $4 N t = c^2 d$, where $d$ is a fundamental discriminant.
		
		\begin{proof}
			By \cite[Proposition 5.7.7]{KRYbook}, we have\footnote{Note that there is an error in the statement of \cite[Proposition 5.7.7]{KRYbook}: the factor $\sqrt 2$ on the right hand side of the equation in that proposition should be $2$.}
			\begin{align}
			\frac{W_{T, \infty} \left(g_v, s, \Phi^{\frac32}_{2,\infty}\right)}{W_{t,\infty}\left(g_{v_2}, s - \frac12, \Phi^{\frac32}_{1,\infty}\right)}  = v_1^{-s/2 + 3/4} (-i) \frac{{2} }{s+1};
			\end{align}

			On the other hand, we may apply \cite[Proposition 14.1]{KRYFaltings} to see
			\begin{equation}
			\frac{W_{t,\infty}\left(g_{v_2}, s - \frac12, \Phi^{\frac32}_{1,\infty}\right)}{W_{t,\infty}\left(g_{v_2}, -s +\frac12, \Phi^{\frac32}_{1,\infty}\right)} = (\pi|t|)^{s -\tfrac12} \cdot
			\begin{cases} \Gamma(\frac{-s + 3}2 ) \Gamma(\frac{s}2 + 1)^{-1}, & t>0 \\[10pt]
			\Gamma(-\frac{s}2) \Gamma(\frac{s-1}2)^{-1}, & t < 0.
			\end{cases}
			\end{equation}
			The lemma follows from a little algebra.
		\end{proof}
	\end{lemma}
	
	\begin{corollary}
		Suppose $T = \psm{0 & \\ & t}$ and $v =\psm{v_1 & \\ & v_2}$. Then
		\begin{equation}
		\begin{aligned}
		W_T(g_v, s, \Phi^{\calL}_{2}) &=  v_1^{-\frac{s}{2} + \frac34} N^{-2s +1 }  \left( \frac{s-1}{s+1} \right)  \frac{\Lambda(-2s+2)}{\Lambda(2s+2)} \left( \prod_{p|N} \beta_p(s)\right)  \\
		& \qquad \times W_{t}(g_{v_2}, -s+\tfrac12, \Phi^{\calL}_1)
		\end{aligned}
		\end{equation}
		where $\Lambda(s) = \prod_{v \leq \infty} \zeta_v(s)$ and
		\begin{equation}
		\beta_p(s) = - \frac{ \zeta_p(s-1) \zeta_p(2s+2) }{\zeta_p(s+1) \zeta_p(-2s+2)}.
		\end{equation}
		\begin{proof}
			Combining  \Cref{lem:Whittaker comparison local} and \Cref{lem:Whittaker comparison arch}, we have
			\begin{equation}
			\begin{aligned}
			\frac{W_{T}(g_{v}, s, \Phi^{\calL}_2) }{W_t(g_{v_2}, -s + \frac12, \Phi^{\calL}_1)} &= v_1^{\frac{s}2 + \frac34}  (-2i)  N^{-s + \frac12} \frac{\Lambda(-2s + 2)}{\Lambda(2s+2)} \\
			& \qquad \times \left( \prod_{p < \infty} \gamma_p(\calV) \chi_{\calV, p}(-1) |2|_p |N|_p^{s-\tfrac12} \right) \cdot \prod_{p|N}  \beta_p(s);
			\end{aligned}
			\end{equation}
			here we used the functional equation $\Lambda(s, \chi_{-d}) = \Lambda(1-s, \chi_{-d})$ for the completed $L$-function $\Lambda(s, \chi_{-d}) = \prod_{v \leq \infty} L_v(s, \chi_{-d})$.
			
			On the other hand, the product formula gives
			\begin{equation}
			\prod_{p<\infty} \gamma_p(\calV) \chi_{\calV, p}(-1) = \gamma_{\infty}(\calV)^{-1}
			\cdot (-1, -N)_{\bbR} = - \gamma_{\infty}(\calV)^{-1}.
			\end{equation}
			By \cite[p.\ 330]{KRYbook}, we have $\gamma_{\infty}(\calV) = -i$, and the corollary follows easily.
		\end{proof}

	\end{corollary}

	\begin{corollary} \label{cor:rank 1 Eis coeff}
		Suppose $T = \psm{0 & \\ & t}$ and $v = \psm{v_1 & \\ & v_2}$. Then
		\begin{equation}
		\begin{aligned}
		v_1^{-3/4}	E'_T(g_v, 0, \Phi_2^{\calL}) &=  2 \, W_t'(g_{v_2}, \tfrac 12, \Phi^{\calL}_1)  \\
		& \qquad + \left( \log v_1 + 2 +  \frac{4\Lambda'(2)}{\Lambda(2)}  + \sum_{p|N} \frac{p-1}{p+1} \log p \right) W_t(g_{v_2}, \tfrac12, \Phi^{\calL}_1).
		\end{aligned}
		\end{equation}
		\begin{proof}
			This follows immediately from the previous corollary, and equations  \eqref{eqn:Eisenstein FC rank 1} and \eqref{eqn:whittaker pullback global}.
			
		\end{proof}
	\end{corollary}
	
	Our next step is to compare this expression with the arithmetic heights of special divisors, as computed in \cite{DuYang}.
	Recall that we have written the Eisenstein in ``classical" coordinates as
	\begin{equation}
	E(\tau, s, \Phi^{\calL}_2) = \det(v)^{-\frac34} E(g_{\tau}, s, \Phi^{\calL}_2) = \sum_T  C_T(v, s, \Phi_2^{\calL}) \, q^T.
	\end{equation}

	\begin{proposition}
		Suppose $T = \psm{0 & \\ & t}$ and $v = \psm{v_1 & \\ & v_2}$. Then
		\begin{equation}
		4  \langle \widehat \calZ(t, v_2), \widehat\omega_N \rangle + \log v_1 \deg Z(t)  =  - \frac{ \prod_{p|N}(p+1)}{12} C_T'(v, 0, \Phi^{\calL}_2),
		\end{equation}
		where $ \langle \widehat \calZ(t, v_2), \widehat\omega_N \rangle  = \widehat \deg (\calZ(t,v_2) \cdot \widehat \omega_N)$ is the intersection pairing, as in \Cref{sec:arith Ch gps}
		
		\begin{proof} This follows  from combining \Cref{cor:rank 1 Eis coeff} with \cite[Theorem 1.3]{DuYang}.
			Let
			\begin{equation}
			\calE_{DY}(\tau, s) =  A_1(s) E(\tau, s-\tfrac12, \Phi^{\calL}_1), \qquad A_1(s) := - \frac{s}{4 \pi} \Lambda(2s)\left(  \prod_{p|N}(1-p^{-2s}) \right) N^{\frac12 + \frac32s}
			\end{equation}
			denote the normalized genus 1 Eisenstein series in \cite{DuYang} (note that the variable $s$ is shifted by 1/2 here, as compared to loc.\ cit.). Then for $\tau_2 = u_2 + i v_2 \in \bbH$,  Theorem 1.3 of loc.\ cit.\ states
			\begin{align}
			\varphi(N) \cdot  \langle \widehat \calZ(t, v_2), \widehat\omega_N \rangle q_2^t &= \calE_{DY,t}'(\tau_2, 1)  - \sum_{p|N} \frac{p}{p-1} \log p \cdot  \calE_{DY,t}(\tau_2, 1) \\
			&= A_1(1) \left( E_t'(\tau_2, \tfrac12, \Phi^{\calL}_1) + \left\{  \frac{A_1'(1)}{A_1(1)} - \sum_{p|N} \frac{p}{p-1} \log p \right\} E_t(\tau_2, \tfrac12, \Phi^{\calL}_1)\right) .
			\end{align}
			Now
			\begin{equation}
			A_1(1) = - \frac{1}{4 \pi} \Lambda(2) N^2 \prod_{p|N} (1- p^{-2}) = - \frac{1}{24} \prod_{p|N} (p^2 -1)
			\end{equation}
			and
			\begin{align}
			\frac{A_1'(1)}{A_1(1)} - \sum_{p|N} \frac{p}{p-1} \log p &= 1 +\frac{2\Lambda'(2)}{\Lambda(2)}  + \frac32 \log N + \sum_{p|N}  \left( \frac{2p^{-2}}{1-p^{-2}}   - \frac{p}{p-1} \right)\log p \\
			&= 	1 +  \frac{2\Lambda'(2)}{\Lambda(2)} + \sum_{p|N} \frac{p-1}{2(p+1)} \log p.
			\end{align}
			%		Finally, for $\tau = iv \in \bbH$, we have $E_t(\tau, s, \Phi^{\calL}_1) = v^{-\frac32} W_{t}(g_{v}, s, \Phi^{\calL}_1 )$.
			
			Thus for $\tau  = \psm{\tau_1 & \\ & \tau_2}$, we have $q_{\tau_2}^t =q^T$ and, comparing with \Cref{cor:rank 1 Eis coeff}, we find
			\begin{align}
			2 	(v_2)^{\frac34}	 \cdot \frac{ \varphi(N)}{A_1(1)} \cdot  & \langle \widehat \calZ(t, v_2), \widehat\omega_N  \rangle  \, q^T \\
			&= 2 W'_t(g_{\tau_2}, \tfrac12, \Phi_1^{\calL})  + 2 \left\{	1 +  \frac{2\Lambda'(2)}{\Lambda(2)} + \sum_{p|N} \frac{p-1}{2(p+1)} \log p\right\}W_t(g_{\tau_2}, \tfrac12, \Phi^{\calL}_1) \\
			&= v_1^{-\frac34} E'_T(g_{\tau}, 0, \Phi^{\calL}_2) - \log v \cdot W_{t}(g_{\tau_2}, \tfrac12, \Phi^{\calL}_1).
			\end{align}
			Finally, \cite[Theorem 1.3]{DuYang} asserts that
			\begin{equation}
			\frac12 \deg(Z(t))q_{\tau_2}^t = \frac{1}{\varphi(N)} \calE_{DY,t}(\tau_2, 1) = \frac{A_1(1)}{\varphi(N)}\cdot v_2^{-\frac34} \cdot W_{t}(g_{v_2}, \tfrac12, \Phi^{\calL}_1).
			\end{equation}
		\end{proof}
	\end{proposition}
	
	We arrive at the main theorem for $T$ of rank 1.
	
	\begin{theorem}
		Suppose $\mathrm{rank}(T) = 1$. Then
		\begin{equation}
		\widehat \deg \calZ(T,v) = \frac{ \prod_{p|N} p+1}{24} C'_T(v, 0, \Phi^{\calL})
		\end{equation}
		
		\begin{proof}
			It follows from definitions that for any $\gamma \in \GL_2(\bbZ)$, we have
			\begin{equation}
			C_{\gamma T {}^t\gamma}(v, s, \Phi^{\calL}) = C_{T}(\gamma^t v \gamma, s, \Phi^{\calL}).
			\end{equation}
			As $\widehat\calZ(T,v)$ satisfies the same invariance, cf.\ \eqref{eqn:calZ SL2 invariance}, we may use \Cref{lem:rank 1 T} to assume that $T = \psm{0 & \\ & t}$. Furthermore, given any $v \in \Sym_2(\bbR)$, one can find a matrix $\theta = \psm{1&*\\& 1} \in \GL_2(\bbR)$ such that $\theta v {}^t\theta$ is diagonal, and we have
			\begin{equation}
			C_{\psm{0 & \\ & t}}(\theta v {}^t \theta, s, \Phi^{\calL}) = 	C_{\psm{0 & \\ & t}}( v, s, \Phi^{\calL}) \qquad \text{and} \qquad \widehat\calZ(\psm{0 & \\ & t},\theta v {}^t \theta  ) = \widehat\calZ(\psm{0 & \\ & t}, v  ).
			\end{equation}
			Thus we may assume $T = \psm{0 & \\ & t}$ and $v = \psm{v_1 & \\ & v_2}$. In this case, we have, by definition,
			\begin{equation}
			\widehat\calZ(T,v) = - 2 \widehat\calZ(t, v_2) \cdot \widehat\omega_N + \left( 0, \log v_1 \delta_{Z(t)} \right),
			\end{equation}
			cf.\ \Cref{def:Z(T) rank 1}. The theorem follows from the previous proposition.
		\end{proof}
	\end{theorem}

	\subsection{$T=0$}
	Our first task is computing the constant term $E_0(g,s,\Phi^{\calL}_2)$. To this end, for $g \in G_{2,\bbA}$, we abbreviate
	\begin{equation}
	B(g,s) := W_0\left(e, s+\tfrac12, (\eta^* \circ r(g)) \Phi^{\calL}_2 \right) , \qquad 	B_v(g,s) := W_{0,v}\left(e, s+\tfrac12, (\eta^* \circ r(g)) \Phi^{\calL}_{2,v} \right)
	\end{equation}
	where $W_0(...)$ is the Whittaker functional in genus 1, and $\eta^*  $ is the map defined in \eqref{eqn:eta defn}.
	Arguing as in \cite[\S 5.9]{KRYbook},  for $g \in G_{2, \bbA}$, we have
	\begin{equation} \label{eqn:E0 decomposition}
	E_0(g,s,\Phi_2^{\calL}) = W_0(g,s, \Phi_2^{\calL}) + \sum_{\gamma \in \Gamma_{\infty} \backslash \mathrm{SL}_2(\bbZ)}  B(m(\gamma) g, s) + \Phi_2^{\calL}(g,s).
	\end{equation}
	For $v \in \Sym_2(\bbR)_{>0}$, take $g = g_z \in G_{2,\bbA}$ as in  \eqref{eqn:g_tau def}.
	
	For the first term, \cite[(5.9.3)]{KRYbook} gives\footnote{Note that the power of 2 appearing on the right hand side of \cite[(5.9.3)]{KRYbook} is misstated.}
	\begin{equation}
	W_{0,\infty}(g_v, s, \Phi_{2,\infty}^{\calL}) = - 2^{3/2} \frac{(s-1) \zeta_{\infty}(2s-1)}{(s+1) \zeta_{\infty}(2s+2)} \det(v)^{-s/2+3/4}
	\end{equation}
	where $\zeta_{\infty}(s) = \pi^{-s/2} \Gamma(s/2)$, and 	
	\begin{equation}
	W_{0,p}(e, s, \Phi^{\calL}_{2,p}) = |2|_p^{3/2}\frac{\zeta_p(2s-1)}{\zeta_p(2s+2)} \gamma(\calV_p)^2
	\end{equation}	
	for $p \nmid N$.
	
	When $p | N$, we may take $a, b \to \infty$ in \Cref{prop:rep dens split}, apply \Cref{lem:rep dens relation} and simplify to find
	\begin{equation}
	W_{0,p}(e, s, \Phi^{\calL}_{2,p}) = \frac{ \zeta_{p}(2s -1)}{\zeta_{p}(2s + 2)} \cdot \frac{p^{-1}(1+p^{-s+1})}{1+p^{-s-1}} \gamma(\calV_p)^2.
	\end{equation}
	Combining these identities with the fact that $\gamma_{\infty}(\calV)^2 = -1$, and applying the product formula $\prod_{v \leq \infty} \gamma_v(\calV) = 1$, we have
	\begin{equation}
	W_0(g_v, s, \Phi^{\calL}_{2}) = \det(v)^{-s/2 + 3/4} \frac{(s-1)\Lambda(2s-1)}{(s+1)\Lambda(2s+2)} \cdot A(s),
	\end{equation}
	where
	\begin{equation}
	A(s) := \prod_{p|N} p^{-1} \,  \frac{1+p^{-s+1}}{1+p^{-s-1}}.
	\end{equation}
	Note that $A(0) = 1$ and
	\begin{equation}
	A'(0) = \sum_{p|N}  \frac{1-p}{1+p} \log p,
	\end{equation}
	and so
	\begin{equation}
	W'_0(g_v, 0, \Phi^{\calL}_2) = (\det v)^{3/4}  \cdot \left( \frac12 \log \det v + 2 - 4 \frac{\Lambda'(-1)}{\Lambda(-1)} - \sum_{p|N} \frac{1-p}{1+p} \log p \right).
	\end{equation}
	
	Next, we consider the middle term in \eqref{eqn:E0 decomposition}.
	For any $\gamma \in \mathrm{SL}_2(\bbZ)$ and prime $p$, we have
	\begin{equation}
	r(m(\gamma)) \Phi^{\calL}_{2,p} = \Phi^{\calL}_{2,p};
	\end{equation}
	hence \Cref{lem:pullback on Rallis sections} gives
	\begin{equation}
	\eta^* \left( r(m(\gamma)) \Phi^{\calL}_{2,p}(s) \right) = \Phi^{\calL}_{1,p}\left( s + \tfrac 12 \right),
	\end{equation}
	which is in particular independent of $\gamma$. Thus we have
	\begin{equation}
	B_p(m(\gamma), s) = W_{0,p}(e, s+\frac12, \Phi^{\calL}_1), \qquad \gamma\in \mathrm{SL}_2(\bbZ).
	\end{equation}
	Applying \eqref{eqn:whittaker rep dens rank 1} and taking $a \to \infty$ in  \eqref{eqn:alpha rank 1} for $p\nmid N$, or  \eqref{eqn:alpha rank 1 level  a odd} and \eqref{eqn:alpha rank 1 level a even} for $p|N$, a short computation gives
	\begin{equation}
	B_p(m(\gamma), s) =
	\chi_{\calV,p}(-1) \gamma_p(\calV) |2|_p^{1/2} \zeta_p(2s+1)
	\begin{cases}
	\zeta_p(2s+2)^{-1}, & p \nmid N \\
	p^{-s-3/2}\zeta_p(s)^{-1}, & p | N	
	\end{cases}
	\end{equation}
	for all $\gamma \in \mathrm{SL}_2(\bbZ)$.
	
	At the infinite place, the proof of \cite[Proposition 5.9.2]{KRYbook} gives
	\begin{equation}
	B_{\infty}(m(\gamma) g_v, s) = \sqrt{2} \, \gamma_{\infty}(\calV) \, \det(v)^{s/2 + 3/4} Q_v(c,d)^{-\frac12 - s} \frac{s\zeta_{\infty}(2s+1)}{(s+1) \zeta_{\infty}(2s+2)}
	\end{equation}
	where $\gamma = \psm{a & b \\ c& d}$ and $Q_v(c,d)  = (c,d) \cdot v \cdot {}^t(c,d)$. Consider the series
	\begin{equation}
	G(s, v) := \sum_{\psm{* & *\\ c& d} \in \Gamma_{\infty} \backslash \mathrm{SL}_2(\bbZ)} Q_v(c,d)^{-s}
	\end{equation}
	for $Re(s) \gg 0$. Note that
	\begin{equation}
	\zeta(2s)	G(s,v) =  Z(s,v)
	\end{equation}
	where
	\begin{equation}
	Z(s,v) := \sum_{\substack{m,n \in \mathbb Z \\ (m,n) \neq (0,0)} }  Q_v(m,n)^{-s}
	\end{equation}
	is the Epstein zeta function attached to the positive definite quadratic form determined by $v$.
	By \cite[\S 1.5]{Siegel}, the series $Z(s,v)$ extends meromorphically  to $s \in \bbC$, with only a  simple pole at $s=1$; in particular, $Z(s, v)$ is regular at $s=1/2$. It follows that $G(s,v)$ vanishes at $s=\frac12$.
	
	The preceding discussion implies that
	\begin{equation} \label{eqn:sum on gamma B term}
	\sum_{\gamma \in \Gamma_{\infty} \backslash \mathrm{SL}_2(\bbZ)} B(m(\gamma) g_v, s) = \det(v)^{s/2 + 3/4} \cdot \frac{s \Lambda(2s+1)}{(s+1)\Lambda(2s+2)} \cdot G(s+\tfrac12, v) \cdot B(s)
	\end{equation}
	where
	\begin{equation}
	B(s) = \prod_{p|N} p^{-s-3/2} \frac{\zeta_p(2s+2)}{\zeta_p(s)}.
	\end{equation}
	Notice that
	\begin{equation}
	\mathop{ord}_{s=0}B(s) = o(N)
	\end{equation}
	where $o(N)$ is the number of prime divisors of $N$, and hence the sum \eqref{eqn:sum on gamma B term} vanishes at $s=0$ to order at least $o(N)+1$.
	
	Finally, by definition we have that
	\begin{equation}
	\Phi_2^{\calL}(g_v, s) = \det(v)^{s/2+3/4}.
	\end{equation}
	
	To summarize the discussion, we obtain:
	
	\begin{proposition} \label{prop:deriv const term}
		Let $C_0(v, s, \Phi^{\calL}_2) =  \det(v)^{-3/4} E_0(g_v, s, \Phi^{\calL}_2)$. Then
		\begin{equation}
		C'_0(v,0, \Phi^{\calL}_2) = \log \det v + 2 - 4 \frac{\Lambda'(-1)}{\Lambda(-1)} - \sum_{p|N} \frac{1-p}{1+p} \log p  .
		\end{equation}
		\qed
	\end{proposition}
	
	\begin{corollary}
		\begin{equation}
		\widehat{\mathrm{deg}} \, \widehat \calZ(0,v) \ = \ \left( \frac{ \prod_{p|N}(p+1)}{24} \right) \cdot C'_0(v,0,\Phi_2^{\calL}).
		\end{equation}
		\begin{proof} Recall that by definition,
			\begin{equation}
			\widehat \calZ(0,v)  = \widehat \omega \cdot \widehat \omega + \left( 0, \log \det v \cdot [ \Omega] \right),
			\end{equation}
			where
			\begin{equation}
			\widehat \omega = - 2 \widehat \omega_N - \sum_{p|N} \widehat \calX^0_p + (0, \log N)  ,
			\end{equation}
			cf.\ \eqref{eqn:widehat L def}, and we set $\Omega = \frac{dx \wedge dy}{2 \pi y^2}$.
			
			The arithmetic degree $ \langle \widehat \omega_N , \widehat \omega_N  \rangle= \widehat \deg \, \widehat \omega_N \cdot \widehat \omega_N$ was computed independently by K\"uhn \cite{Kuhn} and Bost; adjusting for the normalization of the metric \eqref{eqn:metric}, which differs by a multiplicative constant from the normalization of \cite{Kuhn}, and a factor of 2 because our degree is `stacky',  the result is
			\begin{equation}
			\langle \widehat \omega_N , \widehat \omega_N  \rangle = \frac{ \prod_{p|N}(p+1)}{24} \left( \frac12 - \frac{\Lambda'(-1)}{\Lambda(-1)} \right).
			\end{equation}
			
			Next, let $W_N$ denote the Atkin-Lehner involution on $\calX_0(N)$, and note that $W_N^*( \widehat \omega_N )= \widehat \omega_N$ and $W_N^*( \widehat\calX_p^0) =  \widehat \calX_p^{\infty}$, so that
			\begin{equation}
			\langle \widehat \omega_N \cdot \widehat \calX_p^0 \rangle  =  \langle W_N^* \widehat \omega_N , W_N^* \widehat \calX_p^0 \rangle = 	\langle \widehat \omega_N \cdot \widehat \calX_p^{\infty} \rangle   = \frac12 \langle \widehat \omega_N , \widehat \calX_p^{0} + \widehat \calX_p^{\infty} \rangle
			\end{equation}
			On the other hand, we have
			\begin{equation}
			\widehat \calX^0_p + \widehat \calX_p^{\infty} = (\mathcal X_0(N)_{\mathbb F_p}, 0) = \widehat{ \mathrm{div}}(p) + (0, 2\log p) =  (0, 2 \log p) \in \ChowHat{1}(\calX);
			\end{equation}
			here we view $p$ as a rational function on $\calX_0(N)$, and so its arithmetic divisor $\widehat{\mathrm{div}}(p) = (\calX_{/ \mathbb F_p}, - \log p^2)$ vanishes in $\ChowHat{1}(\calX)$.
			
			Thus
			\begin{equation}
			\langle \widehat \omega_N,  \sum_{p|N} \widehat \calX_p^0 - (0, \log n) \rangle = \sum_{p|N} \langle \widehat \omega_N, \calX_p^0 - (0, \log p) \rangle  = 0.
			\end{equation}
			
			Finally, we note that
			\begin{equation}
			\langle \widehat \calX_p^0, (0, \log N) \rangle = \langle (0, \log N), (0, \log N) \rangle = 0
			\end{equation} and $\langle \widehat \calX_p^0,  \calX_q^0 \rangle = 0$ if $p \neq q$. Moreover, by \cite[Lemma 7.2]{DuYang}, we have
			\begin{equation}
			\langle \widehat \calX_p^0, \widehat \calX_p^0 \rangle = - \frac{ \prod_{q|N}(q+1)}{24} \frac{p-1}{p+1} \log p;
			\end{equation}			
			again, this formula differs from \emph{loc.\ cit.} by a factor of 2, because we are using the `stacky' degree.
			
			Putting everything together, we find that
			\begin{equation}
			\langle \widehat \omega, \widehat \omega \rangle = 4 \langle \widehat \omega_N, \widehat \omega_N \rangle + \sum_{p|N} \langle \widehat \calX_p^0, \widehat \calX_p^0 \rangle =   \frac{ \prod_{q|N}(1+q)}{24}   \left( 2 - 4 \frac{\Lambda'(-1)}{\Lambda(-1)}  - \sum_{p|N} \frac{ 1- p}{1+p} \log p \right).
			\end{equation}
			Finally, we observe that
			\begin{equation}
			\widehat{\mathrm{deg}} (0, \log \det v \cdot  [\Omega]) = \frac{\log \det v}{2} \int_{[\Gamma_0(N) \backslash \mathbb H]} \frac{dx \wedge dy}{2 \pi y^2} = \log \det v \cdot \frac{ \prod_{q|N} (q+1)}{24},
			\end{equation}
			and the theorem follows.
		\end{proof}
	\end{corollary}

	\section{Comparison of two Eisenstein series}
	In addition to the lattice $\calL$, there is another lattice which naturally parameterizes the modular curve $X_0(N)$ over $\C$. It is
	$$
	L=\{ A= \begin{pmatrix} a & b \\ N c & -a \end{pmatrix}:\, a, b, c \in \Z\}, \quad Q(A) = \det A =-a^2 - N bc.
	$$
	We set $V = L \otimes_{\bbZ} \bbQ$. Our aim in this section is to compare the incoherent Eisenstein series attached to $L$ with the Eisenstein series $E(\tau, s, \Phi_2^{\calL})$ appearing in our main theorem.

	For $r = 1,2$, let $\Phi^L_r (s) = \otimes_{v \leq \infty} \Phi_{r,v}^L(s) $ be the incoherent standard section in $I_r(s, \chi')$ with $\Phi^L_{r, \infty} = \Phi_{r, \infty}^{3/2}$ and, for $v < \infty$, the component $\Phi^L_{r,v}$ is the standard section attached to $L_v$ via the Rallis map, cf.\ \Cref{def:calL global section}; here the character $\chi'$ is given by $\chi'(x) = (-1, \,x) _\A$.

	We will also require some auxilliary sections. For each prime $\ell$ dividing $N$ let $L^{(\ell)}$ denote the quadratic lattice defined in \eqref{eqn:Lp lattice def}; recall that by definition, $L^{(\ell)}$ is the set of traceless elements in an Eichler order of level $N/\ell$ contained in the definite quaternion algebra ramified exactly at $\ell$ and $\infty$, and the quadratic form is the reduced norm.
	
	For $r=1,2$, let $\Phi^{(\ell)}_r(s) \in I_r(s, \chi')$ denote the  section with $\Phi^{(\ell)}_{r, \infty}(s) = \Phi^{3/2}_{r, \infty}(s)$, and for $v < \infty$, the local section $\Phi^{(\ell)}_{r,v}(s) $ is the standard section associated to $L^{(\ell)}_v$. Note that this section is coherent, in the sense of \cite{KudlaCentralDerivs}.
	
	Finally, for $\ell | N$, we set
	\begin{equation}
	L^{sp, (\ell)}  = \{ A= \begin{pmatrix} a & b \\ (N/{\ell}) c & -a \end{pmatrix}:\, a, b, c \in \Z\}, \quad Q(A) = \det A =-a^2 - (N/\ell) bc.
	\end{equation}
	In the same manner as above, we obtain an incoherent section $\Phi_r^{sp, (\ell)}$ with $\Phi_{r,\infty}^{sp, (\ell)} = \Phi_{r,\infty}^{3/2}$ and where at finite places, $\Phi_{r,v}^{sp, (\ell)}$ is the standard section associated to $L^{sp, (\ell)}_v$.
	
	For convenience, we write
	\begin{equation}
	\calV = \calL \otimes_{\bbZ} \bbQ, \qquad V = L \otimes_{\bbZ} \bbQ, \qquad V^{(\ell)} = L^{(\ell)} \otimes_{\bbZ} \bbQ \qquad \text{and} \qquad V^{sp, (\ell)} = L^{sp, (\ell)} \otimes_{\bbZ} \bbQ.
	\end{equation}

	Our first step is a comparison of Whittaker functionals in genus two.

	\begin{proposition} \label{prop:alt lattice rk 2} Suppose $T \in \Sym_2(\bbZ)^{\vee}$.
		
		\begin{enumerate}[(i)]
			\item If $q \nmid N$, then
			\[
			\frac{W_{T,q}(e, s, \Phi^{\calL}_{2,q})}{\gamma_q(\calV)^{2}} = \frac{W_{NT,q}(e, s, \Phi^{L}_{2,q})}{\gamma_q(V)^{2}} =  \frac{W_{NT,q}(e, s, \Phi^{(\ell)}_{2,q})}{\gamma_q(V^{(\ell)})^{2}}
			\]
			with the last identity holding for any $\ell | N$.
			\item If $q | N$, then
			\[
			\frac{W_{T,q}(e, 0, \Phi^{\calL}_{2,q})}{\gamma_q(\calV)^{2}} = \frac{W_{NT,q}(e, 0, \Phi^{L}_{2,q})}{\gamma_q(V)^{2}} =  \frac{W_{NT,q}(e, 0, \Phi^{(\ell)}_{2,q})}{\gamma_q(V^{(\ell)})^{2}}
			\]
			with the last identity holding for $\ell \neq q$.
			\item If $q|N$, then
			\begin{align*}
			&\frac{ W'_{T,q}(e , 0, \Phi^{\calL}_{2,q})}{\gamma_q(\calV)^2} - \frac{ W'_{NT,q}(e , 0, \Phi^{L}_{2,q})}{\gamma_q(V)^2} \\
			&= \left( \frac{q-1}{2(q+1)} \cdot \frac{W_{NT, q}(e, 0, \Phi_q^{(q)})}{\gamma_q(V^{(q)})^2} + \frac{3 q - 1}{2(q-1)} \frac{W_{NT,q}(e, 0, \Phi_q^L)}{ \gamma_q(V)^2} - \frac{2}{q^2-1} \frac{W_{NT,q}(e, 0, \Phi_q^{sp, (q)})}{\gamma_q(V^{sp,(q)})^2} \right) \log q
			\end{align*}
			\item Let $\tau \in \mathbb H_2$ and $g_{\tau,\infty} \in G_{2, \infty}$ as in \eqref{eqn:g_tau def}. Then for any $\Phi \in I_{2, \infty}(s, \chi'_{\infty})$, we have
			\[
			W_{T, \infty}(g_{N\tau, \infty}, s, \Phi) = N^{-s + \frac32} W_{NT, \infty}(g_{\tau,\infty}, s, \Phi).
			\]
			
			\item  Globally, we have
			\begin{align*}
			N^{-\frac32} \,	W'_T(g_{N\tau}, 0, \Phi^{\calL}_2)  &= W_{NT}'(g_{\tau}, 0, \Phi^L_2) + \left( \sum_{\ell|N}  \frac{\ell + 1}{2(\ell - 1)} \log \ell \right) W_{NT}(g_{\tau}, 0, \Phi^L_2)  \\
			& \qquad + \sum_{\ell|N} \left(  \frac{\ell-1}{2(\ell+1)} W_{NT}(g_{\tau}, 0, \Phi_2^{(\ell)}) - \frac{2}{ \ell^2-1} W_{NT}(g_{\tau}, 0, \Phi^{sp,(\ell)}_2)	\right) \log \ell	\end{align*}
		\end{enumerate}
		\begin{proof}
			For any primes $\ell, q$ with $q \neq \ell$, the lattices $L_q$ and $L^{(\ell)}_q$ are isometric. Hence parts $(i)$ and $(ii)$ of the proposition follow immediately from \Cref{lem:Whittaker values shift}.
			
			To prove \textit{(iii)}, we  work at the level of representation densities. First, a direct computation using  Lemma \ref{lem:rep dens relation} and (\ref{eq:Lp}) gives
			\begin{equation}  \label{eqn:local ram Whittaker 1}
			(X-q) \alpha(X, NT, L_q) + (q^2-q) X^2 \alpha(X, T, \calL_q)
			= (X-1) \left[ q^2 \alpha(X, NT, L^0_{q}) + (q-1) (X-q)\right]
			\end{equation}
			where, changing notation slightly, we let $L^0_{q} = M_2(\Z_q)^{\mathrm{tr} =0}$ with quadratic form
			$Q(x) =\det x$. On the other hand, \cite[Corollary 8.4]{YangLocal} implies that
			\begin{equation} \label{eqn:local ram Whittaker 2}
			q^2 \alpha_q(X, NT, L^0_q) = \left( \frac{1-q}2 \right)\alpha_q(X, NT, L^{(q)}_q)  + \left( \frac{1+q}2 \right)\alpha_q(X, NT, L_q)+ q^2 - 1.
			\end{equation}
			Substituting this into \eqref{eqn:local ram Whittaker 1} and rearranging gives
			\begin{align}
			\left( \frac{X+1}2 \right) \alpha_q(X, NT, L_q) - qX^2\alpha_q(X, T, \calL_q) = \left( \frac{X-1}2 \right) \left[ \alpha_q(X, NT, L_q^{(q)}) - 2(X+1) \right] .
			\end{align}
			Differentiating with respect to $X$ and substituting $X=1$, we have
			\begin{equation} \label{eqn:local ram Whittaker 3}
			\alpha_q'(1, NT, L_q) - q \alpha'_q(1, T,  \calL_q)
			= \frac{ \alpha_q(1, NT, L_q^{(q)}) - \alpha_q(1, NT, L_q)}{2} +2q \alpha_q(1, T, \calL_q) - 2.
			\end{equation}
			To manipulate this expression further, we take $X =1 $ in \eqref{eqn:local ram Whittaker 1} and \eqref{eqn:local ram Whittaker 2}, which yields the relations
			\begin{equation}
			\alpha_q(1, NT, L_q ) = q \alpha_q(1, T, \calL_q)
			\end{equation}
			and
			\begin{align}
			2 &= \frac{2}{q^2-1}	(q^2 -1)   \\
			&= 2 \frac{q^2}{q^2-1} \alpha_q(1, NT, L_q^0)  + \left( \frac{1}{q+1} \right)\alpha_q(1, NT, L^{(q)}_q)  - \left( \frac{1}{q-1} \right)\alpha_q(1, NT, L_q).
			\end{align}
			Substituting these identities back into \eqref{eqn:local ram Whittaker 3} gives
			\begin{align}
			\alpha'_q(1, NT, L_q) - q \alpha'_q(1, T, \calL_q) = \frac{q-1}{2(q+1)}\alpha_q(1, NT, L_{q}^{(q)}) &+ \frac{3q - 1}{2(q-1)} \alpha_q(1, NT, L_q)  \\
			& \qquad	- \frac{2q^2}{q^2-1} \alpha_q(1, NT, L_q^0).
			\end{align}
			Finally, applying \eqref{eqn:Whittaker rep dens relation general}, we obtain  \textit{(iii)}.
			
			To prove \textit{(iv)}, we first observe that for any $b \in \Sym_2(\bbR)$, we have
			\begin{equation}
			w \, n(b) \, g_{N\tau} =  w\, n(b) \, m(\sqrt{N}) g_{\tau} = w \, m(\sqrt{N}) \, n( N^{-1} b) g_{\tau} = m(\sqrt{N^{-1}})\, w \, n(N^{-1}b) g_{\tau},
			\end{equation}
			where the notation is as in \eqref{eqn:Siegel parabolic defns}.
			Thus, for any $\Phi \in I_{2,\infty}(s, \chi'_{\infty})$, we have
			\begin{align}
			W_{T, \infty }(g_{N\tau}, s, \Phi) & = \int_{b \in \Sym_2(\bbR)} \Phi(w n(b) g_{N\tau},s) e^{-2 \pi i \mathrm{tr} Tb} \,  db \\
			& = \int_{b \in \Sym_2(\bbR)} \Phi(m(\sqrt{N^{-1}})\, w \, n(N^{-1}b) g_{\tau},, s) \, e^{-2 \pi i \mathrm{tr} Tb} \, db \\
			&= \chi_{\calV, \infty}(N^{-1}) N^{-s - \frac32}  \int_{b \in \Sym_2(\bbR)} \Phi(w n(N^{-1}b)   g_{\tau}, s) \, e^{-2 \pi i \mathrm{tr} Tb} \, db \\
			&=  N^{-s - \frac32} \cdot N^3  \int_{\Sym_2(\bbR)}  \Phi(w n(b) g_{\tau}, s) \, e^{-2 \pi N \mathrm{tr} Tb} \, db \\
			&= N^{-s + \frac32} W_{NT, \infty}(g_{\tau}, s, \Phi),
			\end{align}
			where in the second-to-last line, we used the fact that $\chi_{\calV, \infty}(\sqrt{N^{-1}}) = (\sqrt{N^{-1}}, - N)_{\infty} = 1$, and applied the change of variables $b \mapsto Nb$.
			
			To prove \textit{(v)}, recall that the local Weil indices, for any quadratic space $W$ over $\bbQ$, satisfy the product formula $\prod_{q \leq \infty} \gamma_q(W)  = 1$. Thus, applying part \textit{(i) -- (iv)} of the proposition, we have
			\begin{align}
			&N^{-\frac32} \,	W'_T (g_{N\tau}, 0, \Phi^{\calL}_2) \notag \\
			&= \left\{ W_{NT}'(g_{\tau}, 0, \Phi^L_2)   + \left( \sum_{\ell|N}  \frac{\ell + 1}{2(\ell - 1)} \log \ell \right) W_{NT}(g_{\tau}, 0, \Phi^L_2) \right\}  \left( \frac{ \gamma_{\infty}(V)}{\gamma_{\infty}(\calV)} \right)^2 \notag\\
			& \qquad + \sum_{\ell|N} \left(  \frac{\ell-1}{2(\ell+1)} W_{NT}(g_{\tau}, 0, \Phi_2^{(\ell)})  \left( \frac{ \gamma_{\infty}(V^{(\ell)})}{\gamma_{\infty}(\calV)} \right)^2 - \frac{2}{\ell(\ell^2-1)} W_{NT}(g_{\tau}, 0, \Phi^{sp,(\ell)}_2) \left( \frac{ \gamma_{\infty}(V}{\gamma_{\infty}(\calV)} \right)^2	\right) \log \ell	
			\end{align}
			Note that $V_{\infty} \simeq \calV_{\infty} \simeq V^{sp, (\ell)}_{\infty}$, and by \cite[p.\ 330]{KRYbook}, we have $\gamma_{\infty}(V) = -i = - \gamma_{\infty}(V^{(\ell)})$. Thus
			\begin{equation}
			\left( \frac{ \gamma_{\infty}(V)}{\gamma_{\infty}(\calV)} \right)^2 = \left( \frac{ \gamma_{\infty}(V^{(\ell)})}{\gamma_{\infty}(\calV)} \right)^2 = 1
			\end{equation}
			and the proposition follows.
		\end{proof}
	\end{proposition}

	To investigate the  degenerate terms, we require the genus one analogue of \Cref{prop:alt lattice rk 2}.
	
	\begin{proposition}  \label{prop:alt lattice rk 1}
		For the lattice $\calL$, let $\calV = \calL \otimes \bbQ$, and for a prime $q \leq \infty$, define
		\[
		c_q(\calV) = \chi_{\calV, q}(-1) \gamma_q(\calV).
		\]
		We define $c_q(V)$, $c_q(V^{(\ell)})$ and $c_q(V^{sp, (\ell)})$ in a similar way.
		
		Let $t \in \mathbb Z$. Then
		\begin{enumerate}[(i)]
			\item Suppose $q \nmid N$ and $q < \infty$. Then
			\[
			\frac{W_{t,q}(e, s, \Phi^{\calL}_{1,q})}{c_q(\calV)} = 					\frac{W_{Nt,q}(e, s, \Phi^{L}_{1,q})}{c_q(V)} = 					\frac{W_{Nt,q}(e, s, \Phi^{(\ell)}_{1,q})}{c_q(V^{(\ell)})}
			\]
			with the last identity holding for any $\ell | N$.
			\item Suppose $q | N$. Then
			\[
			\frac{	W_{t,q}(e, \frac12, \Phi^{\calL}_{1,q})}{c_q(\calV)} =q^{\frac12} \frac{W_{Nt, q}(e,  \frac12, \Phi^{L}_{1,q})}{c_q(V)}  = q^{\frac12} \frac{W_{Nt, q}(e,  \frac12, \Phi^{(\ell)}_{1,q})}{c_q(V^{(\ell)})}
			\]
			with the last identity holding for all $\ell \neq q$.
			\item Suppose $q|N$. Then
			\begin{align*}
			&q^{-\frac12} 	\frac{	W'_{t,q}(e, \frac12, \Phi^{\calL}_{1,q})}{c_q(\calV)}  =	 \frac{	W'_{Nt,q}(e, \frac12, \Phi^{L}_{1,q})}{c_q(V)} \\
			&  +\left( \frac{q+1}{2(q-1)} 	\frac{	W_{Nt,q}(e, \frac12, \Phi^{L}_{1,q})}{c_q(V)} - \frac{q-1}{2(q+1)} 	\frac{	W_{Nt,q}(e, \frac12, \Phi^{(q)}_{1,q})}{c_q(V^{(q)})}  - \frac{2}{q^2 - 1}   	\frac{	W_{Nt,q}(e, \frac12, \Phi^{sp,(q)}_{1,q})}{c_q(V^{sp,(q)})}\right) \log q
			\end{align*}
			\item Let $\tau = u + iv \in \mathbb H$ and $g_{\tau,\infty} = [n(u) m(\sqrt v), 1] \in G_{1,\infty}$ denote the corresponding group element. Then for any $\Phi \in I_{1, \infty}(s, \chi'_{\infty})$, we have
			\[
			W_{t, \infty}(g_{N\tau, \infty}, s, \Phi) = N^{-\frac{s}2 + \frac12} W_{Nt, \infty}(g_{\tau,\infty}, s, \Phi).
			\]
			
			\item  Globally,
			\begin{align*}
			N^{-3/4}W_{t}'(g_{N\tau}, \tfrac12, \Phi^{\calL}_1)  &= W_{Nt}'(g_{\tau}, \tfrac12, \Phi^L_1) + \left( \sum_{\ell|N} \frac{\log \ell }{\ell-1} \right)W_{Nt}(g_{\tau}, \frac12, \Phi^L_1) \\
			& \qquad  + \sum_{\ell|N} \left( \frac{\ell-1}{2(\ell+1)} W_{Nt}(g_{\tau}, \tfrac12, \Phi^{(\ell)}_1) - \frac{2}{\ell^2 -1 } W_{Nt}(g_{\tau}, \tfrac12, \Phi^{sp, (\ell)}_1) \right) \log \ell
			\end{align*}
		\end{enumerate}
		\begin{proof}
			First, suppose that $q \nmid N$. The representation density $\alpha_q(X, t, \calL_q)$ was computed previously, cf.\  \eqref{eqn:alpha rank 1}. On the other hand, we have that for any $\ell | N$, the quadratic spaces $L_q$ and $L^{(\ell)}_q$ are both isometric to $(M_2(\bbZ_q)^{tr=0}, \det)$. Applying \cite[ Proposition 8.3]{YangLocal}, a short computation reveals
			\begin{equation}
			\alpha_q(X, t, \calL_q) = \alpha_q(X, Nt, L_q) = \alpha_q(X, Nt, L^{(\ell)}_q).
			\end{equation}
			Finally, applying the identity \eqref{eqn:whittaker rep dens rank 1}, which relates the representation densities to Whittaker functionals, yields \textit{(i)}.
			
			When $q|N$, the representation density $\alpha_q(X, t, \calL_q)$ can be computed explicitly using  \cite[Theorem 3.1]{YangLocal}; the result is recorded in equations \eqref{eqn:alpha rank 1 level a odd} and \eqref{eqn:alpha rank 1 level a even} above. Comparing these formulas with the explicit formulas in Corollary 8.2 and Proposition 8.3 of \cite{YangLocal}, we find
			\begin{equation} \label{eqn:alt alpha q divides N}
			\alpha_q(X, t, \calL_q) = 1 - X^{-1} + X^{-1} \alpha_q(X, Nt, L_q)
			\end{equation}
			Evaluating at $X=1$ gives
			\begin{equation}
			\alpha_q(1, t, \calL_q) = \alpha_q(1, Nt, L_q) = \alpha_q(1, Nt, L_q^{(\ell)})
			\end{equation}
			for $\ell \neq q$, and hence, applying \eqref{eqn:whittaker rep dens rank 1} again gives part \textit{(ii)} of the proposition.
			
			To prove \text{(iii)} we differentiate  \eqref{eqn:alt alpha q divides N}  and evaluate at $X=1$, which yields
			\begin{equation}
			\alpha'_q(1, t, \calL_q) = \alpha' _q(1, Nt, L_q) - \alpha_q(1, Nt, L_q) + 1.
			\end{equation}
			
			By \cite[Corollary 8.2]{YangLocal}, we have $\alpha_q(X, Nt, L_q) + \alpha_q(X, Nt, L_q^{(q)}) = 2$; substituting this identity above, we conclude that
			\begin{equation}
			\alpha'_q(1, t, \calL_q) = \alpha' _q(1, Nt, L_q) - \frac12 \alpha_q(1, Nt, L_q) + \frac12 \alpha_q(1, Nt, L^{(q)}_q).
			\end{equation}
			
			On the other hand, the same corollary implies that
			\begin{equation}
			\frac{2q }{q^2-1} \alpha_q(1, Nt, L^0_q) -  \frac{\alpha_q(1, Nt, L_q )}{q-1} -  \frac{\alpha_q(1, Nt, L_q^{(q)})}{q+1} = 0.
			\end{equation}
			Thus
			\begin{align}
			\alpha'_q(1, t, \calL_q) - \alpha'_q(1, Nt, L)  &= - \frac{ q+1}{2(q-1)} \alpha_q(1, Nt, L_q) + \frac{q-1}{2(q+1)} \alpha_q(1, Nt, L^{(q)}_q) \\
			& \qquad\qquad  + 	\frac{2q }{q^2-1} \alpha_q(1, Nt, L^0_q).
			\end{align}
			Another appeal to \eqref{eqn:whittaker rep dens rank 1} gives part \textit{(iii)} of the proposition.
			
			Part \textit{(iv)} follows in the same way as the proof of \Cref{prop:alt lattice rk 2}(iv).
			
			Finally, we observe that for any quadratic space $W$, the constants $c_q(W)$ satisfy the product formula $\prod_{q \leq \infty} c_q(W) = 1$. Combining this observation with the preceding parts of the proposition, a short computation yields
			\begin{align}
			N^{-\frac32} W_{t}'(g_{\tau}, \frac12, \Phi^{\calL}) &= \left( W_{Nt}'(g_{\tau}, \tfrac12, \Phi^L_1) + \left( \sum_{\ell|N} \frac{\log \ell }{\ell-1}  \notag \right)W_{Nt}(g_{\tau}, \frac12, \Phi^L_1) \right) \frac{c_{\infty}(V)}{c_{\infty}(\calV)} \\
			& \qquad  + \sum_{\ell|N} \left(  - \frac{\ell-1}{2(\ell+1)} W_{Nt}(g_{\tau}, \tfrac12, \Phi^{(\ell)}_1) \frac{c_{\infty}(V^{(\ell)})}{c_{\infty}(\calV)}  \right) \log \ell \notag \\
			& \qquad+ \left( - \frac{2}{\ell^2 -1 } W_{Nt}(g_{\tau}, \tfrac12, \Phi^{sp, (\ell)}_1) \frac{c_{\infty}(V^{sp,\ell})}{c_{\infty}(\calV)} \right) \log \ell
			\end{align}
			Since $\calV_{\infty} \simeq V_{\infty} \simeq V^{sp, (\ell)}_{\infty}$, we have$c_{\infty}(\calV) = c_{\infty}(V)c_{\infty}(V^{sp, (\ell)}) . $
			Moreover,  we have
			\begin{equation}
			\chi_{\calV, \infty}(-1) = (-1, -1 )_{\infty} = -1 = \chi_{\calV^{(\ell)}, \infty}(-1) \qquad \text{ and } \qquad \gamma_{\infty}(\calV) = - \gamma_{\infty}(V^{(\ell)})
			\end{equation}
			cf.\ \cite[p. 330]{KRYbook}. Thus $c_{\infty}(V^{\ell}) = - c_{\infty}(\calV)$, and the proposition follows.
			
		\end{proof}
	\end{proposition}
	
	For a section $\Phi \in I_{2}(s, \chi')$, let  $E(g, s, \Phi) = \sum_{T} E_T(g, s, \Phi)$ denote the Fourier expansion of the corresponding Eisenstein series, as in \Cref{sec:main theorem setup}.
	
	\begin{proposition} \label{prop5.3}
		Let $\tau \in \mathbb H_2$. Then for any $T \in \Sym_{2}(\bbZ)^{\vee}$ we have
		\begin{equation}
		N^{- \frac32}E'_T(g_{N\tau}, 0, \Phi^{\calL}_2) = E'_{NT}(g_{\tau}, 0, \Phi^{L}_2) + \sum_{\ell|N} \frac{\ell-1}{2(\ell+1)} E_{NT}(g_\tau 0, \Phi^{(\ell)}_2) \log \ell.
		\end{equation}
	
		\begin{proof}
			
			We note that in general, if $\tau = u+iv\in \mathbb H_2$, we have
				\begin{equation}
					E_{T}(g_{\tau},s, \Phi) = e^{2 \pi i tr(Tu)} E_T(g_v, s, \Phi).
				\end{equation}
			Thus, we may assume that $\tau = iv$ for $v \in \Sym_2(\bbR)_{>0}$, and write $g_{\tau} = g_v$ in the remainder of the proof.

			Now suppose $T$ is non-degenerate. Then for any section $\Phi \in I_2(s, \chi')$, we have
			\begin{equation}
			E_T(g, s, \Phi) = W_T(g, s, \Phi),
			\end{equation}
			cf.\ \Cref{prop:Whittaker dichotomy}. On the other hand, recall that Eisenstein series $E(g, s, \Phi^L_2)$ and $E(g, s, \Phi^{sp, (\ell)}_2)$ are incoherent, in the sense of \cite{KudlaCentralDerivs}, and so they vanish at $s=0$. In particular,
			\begin{equation}
			W_{NT}(g_v, 0, \Phi^{L}) = 0  \qquad \text{and} \qquad   	W_{NT}(g_v, 0, \Phi^{sp, (\ell)}) = 0  .
			\end{equation}
			Thus, the proposition follows from \Cref{prop:alt lattice rk 2}(v) in this case.
			
			Next, suppose that $T$ has rank 1. By \Cref{lem:rank 1 T}, we may choose $\gamma \in \GL_2(\bbZ)$ that $T = {}^t \gamma \psm{0 & \\ & t } \gamma$ for some $t \neq 0$.  Since
			\begin{equation}
			E_{T}(g, s, \Phi) = E_{ {}^t \gamma \psm{0 & \\ & t } \gamma}(g, s, \Phi) = E_{ \psm{0&\\ & t}}  (m(\gamma) g, s, \Phi)
			\end{equation}
			for any section $\Phi \in I_2(s, \chi')$, we may assume without loss of generality that $T = \psm{0  & \\ & t}$. Similarly, replacing $v$ with $\theta v {}^t \theta$ for an appropriate choice of $\theta = \psm{1 & * \\ & 1} \in \GL_2(\bbR)$, we may further assume that
			\begin{equation}
			v = \psm{v_1 & \\ & v_2}
			\end{equation}
			is diagonal.

			For $r = 1$ or $2$, let $\Phi^*_r$ denote any one of the sections $\Phi^{\calL}_r$, $\Phi^{L}_r$, $\Phi^{(\ell)}_r$	or $\Phi^{sp, (\ell)}_r$. Note that in all cases, $\Phi^*_{r, \infty} = \Phi^{\frac32}_{r, \infty}$. In this case, combining \Cref{lem:GS Eis FC rk 1}, \Cref{lem:pullback on Rallis sections} and \eqref{eqn:arch section pullback}, we find
			\begin{equation} \label{eqn:alt section FC rk 1}
			E_T(g_v, s, \Phi^*_2) = v_1^{\frac{s}2 + \frac34}W_{t}\left(g_{v_2}, s + \tfrac12,   \Phi^*_1 \right)  + W_{T}(g_v, s, \Phi^*_2).
			\end{equation}
			In particular,
			\begin{multline}
			N^{-\frac32}	E'_{T}(g_{Nv}, 0 , \Phi^{\calL}_2) =   v_1^{\frac34} \left( \frac{\log (N  v_1)}{2} \right) N^{-\frac34} W_{t}(g_{Nv_2}, \tfrac12, \Phi^{\calL}_1) +   v_1^{\frac34} N^{-\frac34}W'_t(g_{Nv_2}, \tfrac12, \Phi^{\calL}_1) \\
		  + N^{-\frac32} W'_{T}(g_v,0, \Phi^{\calL}_2).
			\end{multline}
			Applying \eqref{eqn:alt section FC rk 1}, \Cref{prop:alt lattice rk 2} and \Cref{prop:alt lattice rk 1}, a short computation gives
			\begin{align*}
			N^{- \frac32}& E'_T(g_{Nv}, 0, \Phi^{\calL}_2)  - E'_{NT}(g_v, 0, \Phi^{L}_2) - \sum_{\ell|N} \frac{\ell-1}{2(\ell+1)} E_{NT}(g_v, 0, \Phi^{(\ell)}_2) \log \ell \\
			&=  \left( \sum_{\ell|N} \frac{\ell + 1}{2(\ell - 1)} \log \ell  \right) \left\{v_1^{\frac34} W_{Nt}(g_{v_2}, \tfrac12, \Phi^L_1)  + W_{NT}(g_v, 0, \Phi^L_2 )\right\}  \\
			& \qquad -    \sum_{\ell|N} \frac{2}{\ell^2-1}  \left\{ v_1^{\frac34}W_{Nt}(g_{v_2}, \tfrac12, \Phi^{sp,(\ell)}_1)  + W_{NT}(g_v, 0, \Phi^{sp,(\ell)}_2) \right\}\log \ell\\
			& =\left( \sum_{\ell|N} \frac{\ell + 1}{2(\ell - 1)} \log \ell  \right) \left\{  E_{NT}(g_v, 0, \Phi^L_2 )\right\}   -    \sum_{\ell|N} \frac{2}{\ell^2-1}  \left\{E_{NT}(g_v, 0, \Phi^{sp,(\ell)}_2) \right\}\log \ell\\
			&= 0
			\end{align*}
			where the last line follows since the sections $\Phi^L_2$ and $\Phi_2^{sp, (\ell)}$ are incoherent, so the Eisenstein series $E(g, s, \Phi^L_2)$ and $E(g, s, \Phi_2^{sp, (\ell)})$ vanish at $s=0$. This proves the proposition in the case that $T$ has rank 1.
			
			Finally, suppose $T= 0$. Recall that for any section $\Phi_2 \in I_2(s, \chi)$, an argument along the lines of \cite[\S 5.9]{KRYbook} (see also \cite[Lemma 2.4]{KudlaRallisSW}) gives
			\begin{equation}
			E_0(g, s, \Phi) = \Phi(g,s) + \bbB(g, s, \Phi) + W_0(g, s, \Phi)
			\end{equation}
			where
			\begin{equation}
			\bbB(g, s, \Phi) = \sum_{\gamma \in \Gamma_{\infty} \backslash \mathrm{SL}_2(\bbZ)} B(m(\gamma)g, s, \Phi)
			\end{equation}
			with
			\begin{equation}
			B(g, s, \Phi) = W_0\left(e, s + \tfrac12, \left( \eta^* \circ r(g) \right) \Phi \right).
			\end{equation}
			A similar argument to the proof of the \Cref{prop:deriv const term} shows that
			\begin{equation}
			\bbB(g_v, 0, \Phi^{(\ell)}) = \bbB(g_v, 0, \Phi^{sp, (\ell)}) = 0
			\end{equation}
			and
			\begin{equation}
			\ord_{s=0} \bbB(g_v, s, \Phi^L) \geq 2 .
			\end{equation}
			Moreover, if $\Phi^*_2$ is any of the sections $\Phi^{\calL}_2$, $\Phi^{L}_2$, $\Phi^{(\ell)}_2$ or $\Phi^{sp, (\ell)}_2$, the definitions  imply the formula
			\begin{equation}
			\Phi^{*}_2(g_v, s)= \det(v)^{\frac{s}2 + \frac34}.
			\end{equation}
			Thus, we find
			\begin{equation}
			E'_0(g_{Nv}, 0, \Phi^\calL_2) = N^{\frac32} \det(v)^{\frac34} \left( \log N + \frac12 \log \det v \right) + W'_0(g_{Nv}, 0, \Phi_2^{\calL}),
			\end{equation}
			and
			\begin{equation}
			E'_0(g_{ v}, 0, \Phi^L_2)  =  \det(v)^{\frac34} \left(  \frac12 \log \det v \right) + W'_0(g_{v}, 0, \Phi_2^{L});
			\end{equation}
			in addition, since $\Phi^L_2$ and $\Phi^{sp,(\ell)}$ are incoherent, we have
			\begin{equation}
			E_0(g_v, 0, \Phi_2^{L}) = E_0(g_v, 0, \Phi_2^{sp,(\ell)}) =  0
			\end{equation}
			so
			\begin{equation}
			W_{0}(g_v, 0, \Phi^{L}) = W_0(g_v,0,\Phi^{sp,(\ell)}) = - \det(v)^{\frac34}.
			\end{equation}
			Thus,  \Cref{prop:alt lattice rk 2} implies that
			\begin{align}
			N^{- \frac32} &	E'_0(g_{Nv}, 0, \Phi_2^{\calL} )  - E'_0(g_v, 0, \Phi_2^L) - \sum_{\ell|N} \frac{\ell-1}{2(\ell+1)} E_0(g_v,0,\Phi^{(\ell)}_2) \log \ell \\
			&= \det(v)^{\frac34} \log N+ \left( \sum_{\ell|N} \frac{\ell+1}{2(\ell-1)} \log \ell \right)W_0(g_v, 0 , \Phi_2^L)   \notag \\
			& \qquad - \sum_{\ell} \frac{2}{\ell^2-1} W_0(g_v, 0, \Phi_2^{sp,(\ell)}) \log \ell- \sum_{\ell} \frac{\ell-1}{2(\ell+1)} \log \ell \cdot \det(v)^{\frac34} \\
			&= \det(v)^{\frac34} \left( \sum_{\ell} 1 - \frac{\ell+1}{2(\ell-1)} + \frac{2}{\ell^2 -1} - \frac{\ell-1}{2(\ell+1)} \right) \log \ell \\
			&= 0 \notag.
			\end{align}
			This implies the proposition for the case $T=0$.
		\end{proof}
	\end{proposition}
	
	\begin{remark}
		The preceding proposition can be phrased in more classical language, as follows, cf.\  \Cref{sec:main theorem setup}. If $\Phi^*_2$ is one of the sections $\Phi^{\calL}_2, \Phi^{L}_2$ or $\Phi^{(\ell)}_2$, and $\tau \in \bbH_2$, we set
		\begin{equation}
		E(\tau, s, \Phi^*_2) = \det (v )^{-\frac34} E(g_{\tau}, s, \Phi^*_2).
		\end{equation}

		Let
			\begin{equation}
				\Gamma_0(4N) = \left\{  \begin{pmatrix} A & B \\ C & D \end{pmatrix} \in \Sp_r(\bbZ) \ | \ C \equiv 0 \pmod{4N} \right\}
			\end{equation}
		denote the usual congruence subgroup. Then it can be verified that $E(\tau, s, \Phi^*_2)$ transforms like a Siegel modular form of scalar weight $3/2$ and level $\Gamma_0(4N)$ with character $\chi(\gamma) = \mathrm{sgn}\det (d)$, see \cite[\S 8.5.6]{KRYbook} for a more precise formulation.
	
		Let $M_{\frac32}(\Gamma_0(4N))$ denote the space of (non-holomorphic) Siegel modular forms of scalar weight $3/2$ and level $\Gamma_0(4N)$ and character $\chi$, and consider the Hecke operator
		\begin{equation}
		U_N \colon M_{\frac32}(\Gamma_0(4N)) \to M_{\frac32}(\Gamma_0(4N))
		\end{equation}
		given by
		\begin{equation}
		U_N (F)(\tau) = \sum_{u \in \Sym_2(\bbZ/N\bbZ)} F|_{3/2} \psm{1 & u \\ & N} = N^{\frac32} \sum_u F(N^{-1}  \tau + N^{-1} u)
		\end{equation}
		Then the previous proposition can be rewritten in classical terms as the identity
		\begin{equation} \label{eqUp1}
		E'(\tau,0,\Phi^{\calL}_2) = U_N \left( E'(\tau, 0, \Phi^L_2) + \sum_{\ell|N} \frac{\ell - 1}{2 (\ell + 1)} E(\tau,0, \Phi^{(\ell)}) \log \ell \right).
		\end{equation}
		
	\end{remark}

	\bibliographystyle{alpha}
	\bibliography{refs}

	\author{\noindent \small \textsc{Department of Mathematics, University of Manitoba} \\
		{\it E-mail address:} \texttt{siddarth.sankaran@umanitoba.ca}}

	\author{\noindent \small \textsc{Department of Mathematics, University of Wisconsin, Madison} \\
		{\it E-mail address:} \texttt{shi58@wisc.edu}}

	\author{\noindent \small \textsc{Department of Mathematics, University of Wisconsin, Madison} \\
		{\it E-mail address:} \texttt{thyang@math.wisc.edu}}
	
\end{document}